\documentclass[11pt,twoside]{article}
\usepackage{latexsym,bbold}
\usepackage{amssymb,amsbsy,amsmath,amsfonts,amssymb,amscd,mathtools}
\usepackage{graphicx}
\usepackage{hyperref}
\usepackage[dvipsnames]{xcolor}
\usepackage{pstricks}
\newcommand{\commentout}[1]{}

\newcommand{\R}{\mathbb{R}}

\newcommand {\e}  {\varepsilon}

\newcommand {\Chi} {{\bf \raise 2pt \hbox{$\chi$}} }

\newcommand {\f}   {\frac}
\newcommand {\p}   {\partial}

\newcommand{\MH}{}

\newcommand{\fer}{\eqref}
\newcommand{\beq}{\begin{equation}}
\newcommand{\eeq}{\end{equation}}
\newcommand{\bea} {\begin{array}{rl}}
\newcommand{\eea} {\end{array}}
\newcommand{\bepa}{\left\{ \begin{array}{l}}
\newcommand{\eepa} {\end{array}\right.}
\newtheorem{theorem}{Theorem}
\newtheorem{lemma}[theorem]{Lemma}
\newtheorem{definition}[theorem]{Definition}

\newtheorem{proposition}[theorem]{Proposition}

\newcommand{\qed}{{ \hfill
                       {\unskip\kern 6pt\penalty 500 \raise -2pt\hbox{\vrule\vbox to 6pt{\hrule width 6pt
                       \vfill\hrule}\vrule} \par}   }}
\title{A moment-based approach for the analysis of the infinitesimal model in the regime of small variance
 }
\author{J. Guerand\thanks{Institut Montpelli\'erain Alexander Grothendieck, Univ. Montpellier, CNRS, Montpellier, France; E-mail: \texttt{jessica.guerand@umontpellier.fr} } \and M. Hillairet\thanks{Institut Montpelli\'erain Alexander Grothendieck, Univ. Montpellier, CNRS, Montpellier, France; E-mail: \texttt{matthieu.hillairet@umontpellier.fr}} \and S. Mirrahimi\thanks{ Institut Montpelli\'erain Alexander Grothendieck, Univ. Montpellier, CNRS, Montpellier, France; E-mail: \texttt{sepideh.mirrahimi@umontpellier.fr}}}

\date{\today}

\begin{document}
\maketitle
\pagestyle{plain}
\pagenumbering{arabic}

\abstract{We provide an asymptotic analysis of a nonlinear integro-differential equation describing the evolutionary dynamics of a population which reproduces sexually and which is subject to selection and competition. The sexual reproduction is modeled via a nonlinear integral term, known as the ''infinitesimal model''. We consider a regime of small segregational variance,   where a parameter in the infinitesimal operator, which measures the deviation between the trait of the offspring and the mean parental trait, is small. We prove  that, in this regime, the phenotypic distribution remains close to a Gaussian profile with a fixed small variance and we characterize the dynamics of the mean phenotypic trait via an ordinary differential equation. While similar properties were already proved for a closely related model using a Hopf-Cole transformation and perturbative analysis techniques, we provide an alternative proof which  relies on a direct study of the dynamics of the moments of the phenotypic distribution and a contraction property of the Wasserstein distance.    }

\medskip

\noindent {\em 2020 M.S.C. :} 35B40, 35Q92, 92D15, 47G20.

\medskip

\section{Introduction}
\label{sec:intro}

\subsection{Model and question}

The purpose of this article is to provide an asymptotic analysis for small $\varepsilon$ of the following equation
\begin{equation} \label{eq_main}
\begin{cases}
\varepsilon^2 \partial_t   n_\e = r T_{\varepsilon}  [ n_\e]  
-   \left( m  + \kappa \rho_\e(t) \right) n_\e ,\\
\rho_\e(t)= \int_{\mathbb R}   n_\e(t,y)dy,\\
n_\e(0,x)=n_{\e,0}(x),
\end{cases}
\end{equation}
where
\[
T_{\varepsilon}[n_\e](x) =\int_{\mathbb R} \int_{\mathbb R} \Gamma_\e \left(x-\frac{(y+y')}{2}\right)  n_\e(t,y)\f{n_\e(t,y')}{\rho_\e(t)} dydy',
\]
\[
\Gamma_\e(x )= \f{1}{\e\sqrt{\pi}} \exp\left(-\frac{x^2}{\e^2}\right).
\]
This equation describes the evolutionary dynamics of a phenotypically structured population, which reproduces sexually and which is subject to selection and competition. The function $n_\e(t,x)$ stands for the density of a population characterized by a phenotypic trait $x$ and $m(x)$ represents a trait-dependent mortality rate. The population is also subject to a mortality rate due to uniform competition between individuals, represented by the term $\kappa \rho_\e$, where $\rho_\e$ is the total size of the population. The parameter $r >0$ scales reproduction in the population.

 The term $T_\e$ describes the sexual reproduction and is based on the assumption that  the trait of an offspring has a normal distribution with a fixed variance, centered around the mean parental trait. This reproduction model, which is usually referred to as the "infinitesimal model", was first suggested by R. A. Fisher in \cite{RF:19} and it has been widely used in the biological literature   \cite{MB:80,Lynch-Walsh,MT:17}.  Its validity is  under the assumption that the trait $x$ is coded by a large number of alleles which have small and additive effects (see \cite{NB.AE.AV:17} for a recent rigorous derivation).
The parameter $\e$ has been introduced in the problem above to consider a small segregational variance, that is the deviation between the trait of an offspring from the mean parental trait is small. This assumption is closely related to the so-called "weak selection regime" in Quantitative Genetics \cite{RB:00,MT:17}.

We will prove that, under appropriate assumptions and for $\e$ small, the phenotypic density $n_\e$ remains close to a Gaussian distribution with variance $\e^2$ and with a mean value which varies according to  an ordinary differential equation.  Note that in many biological articles, when considering the infinitesimal model, it is assumed that the phenotypic distribution is a  Gaussian with constant variance (see for instance \cite{RL.SS:96,MK.NB:97,OR.MK:01,MK.SM:14}).  Our result   justifies this assumption, considering a homogeneous environment and a small segregational variance. Such a result, considering a closely related model without a competition term and with different assumptions, was already provided in \cite{FP:23} using a perturbative analysis via a Hopf-Cole transformation, that is
$$
n_\e(t,x)=\f{1}{\sqrt{2\pi}\e}\exp \left(\f{u_\e(t,x)}{\e^2}\right).
$$
 Here, we provide an alternative proof and we show that a rather straightforward analysis of the moments, together with a contraction property of the Wasserstein distance, leads to the result. 
 
  \subsection{Assumptions and notations}

We assume that the mortality rate $m \in C^{\infty}(\mathbb R)$ enjoys the following properties for some $L >0$:
\begin{itemize}
\item[(H0)] $m(x)\geq 0$ with $\inf_{x\in \R}m(x) = 0$, 
\item[(H1)] $m'(0)  = 0$ and $A_0  \leq m''(x)$ for all $x\in (-L,L)$,  for some constant $A_0 >0,$ 
\item[(H2)] $\max_{x \in [-L,L]} m(x) < r $,
\item[(H3)] $|m^{''}(x)| \leq A_m  (1+|x|^{p})$ for all $x\in \R$ and some exponent $p \in \mathbb N^*$ and constant $A_m \in (0,\infty).$
\end{itemize}
In assumption (H0), we fix that the infimum value of  $m$ is $0.$ Indeed,  in \eqref{eq_main} we can always substract a constant to $m$ up to multiply $n_{\e} $ with a suitable time-exponential.
Assumption (H1) states that the trait $x=0$ is a non-degenerate local minimum point of the mortality rate $m$.  Assumption (H2) ensures that the value of $m$ around this minimum point is not too far from its global minimum.  Assumption (H3)  imposes a technical limitation on the growth at infinity.  These assumptions entail in particular that:
\begin{equation} \label{eq_bis}
\left\{
\begin{aligned}
A_0 |x| & \leq  |m'(x)| ,  & A_0 \dfrac{|x|^2}{2} & \leq m(x) { - m(0)},  && \forall \, x \in (-L,L) ,\\
|m'(x)| & \leq {A_m\left(|x|+\dfrac{|x|^{p+1}}{p+1}\right)}, & 
m(x) {- m(0)} & \leq  {A_m\left(\dfrac{|x|^2}{2}+\dfrac{ |x|^{p+2}}{(p+1)(p+2)}\right)} \,, && \forall \, x \in \mathbb R.
\end{aligned}
\right.
\end{equation}
We comment on the global-in-time well-posedness of \eqref{eq_main} under assumption (H0)-(H3) below.  Given a solution $n_{\e}$ to \eqref{eq_main}, let's denote 
\beq
\label{def:qe}
q_\e(t,x)=\frac{n_\e(t,x)}{\rho_\e(t)}.
\eeq
One can verify that $q_\e$ solves the following equation 
\begin{equation} \label{eq_rm}
\varepsilon^2 \partial_t q_\e = r \left(\widetilde T_{\varepsilon}[q_\e] - q_\e \right) 
-  \left( m  - \int_{\mathbb R} mq_\e dx \right) q_\e ,
\end{equation}
with
\[
\widetilde T_{\varepsilon}[q_\e](x) =  \int_{\mathbb R} \int_{\mathbb R} \Gamma_\e \left(x-\frac{(y+y')}{2}\right) q_\e(t,y)q_\e(t,y') dydy'.
\]
We remark that \eqref{eq_rm} preserves  probability densities.   We next introduce the following notations corresponding to the moments of the phenotypic distribution:
\beq
\label{moments}
\begin{aligned}
& M_{\e,1}(t) = \int_{\mathbb R} x q_\e(t,x){\rm d}x, \\  
& M_{\e,k}^{c} (t)= \int_{\mathbb R} (x- M_{\e,1}{(t)})^{k} q_\e(t,x){\rm d}x, \quad
M_{\e,k}^{|c|}(t) = \int_{\mathbb R} |x- M_{\e,1}{(t)}|^{k} q_\e(t,x){\rm d}x  \quad \forall \, k  \in \mathbb N.
\end{aligned}
\eeq
We remark that the above convention entails in particular $M_{\e,0}^{c} = 1$ and $M_{\e,1}^c =0$ in the cases $k=0,1.$ We finally define
\begin{equation} \label{eq_gepsilon}
g_{\varepsilon}(t,x) = \dfrac{1}{\sqrt{2\pi}\varepsilon} \exp \left(-  \dfrac{(x-\bar{Z}_\e(t))^2}{2\varepsilon^2}\right)
\end{equation}
with   $\bar{Z}_\e$ the solution to:
\begin{equation}
\label{eq:Ze}
\left\{
\begin{aligned}
\dot{\bar{Z}}_\e(t) &= -  m'(\bar{Z}_\e(t)), \\
\bar{Z}_\e(0) & = M_{\e,1}(0).  
\end{aligned}
\right.
\end{equation}

\medskip

We will make the following assumptions guaranteeing a well-prepared initial condition $q_{\varepsilon,0}$ (or equivalently $n_{\varepsilon,0}$):
\begin{equation} \label{eq_initialdata}
\tag{H4}
M_{\varepsilon,1}(0) \in (-L,L), \qquad
 \varepsilon^{-2k_0} M_{\e,2k_0}^c(0) \leq C_1,
\end{equation}
with  $C_1 > 0$ fixed  and  $k_0$ large enough such that
\begin{equation} \label{eq_aprioriM1}
k_0> 3  + \lceil p/2  \rceil,\qquad  r  \left( 1 - \dfrac{2}{4^{k_0}} \right)  -  \max_{x\in[-L,L]}m (x) = : \eta  > 0 .
\end{equation}
Both parameters $C_1$ and $k_0$ shall be chosen independent of $\varepsilon.$ To address the limiting behavior when $\varepsilon \to 0$ we also assume that
\begin{align}
\tag{H5}
\label{as:M10}
& M_{\e,1}(0)= x_0+O(\e), &&  x_0\in  (-L,L) ,
\\[4pt]
\tag{H6}
\label{as:rho0}
& \rho_m\leq \rho_\e(0) \leq \rho_M, && \text{for some positive constants $\rho_m$ and $\rho_M$}.
\end{align}
{We point out that Assumption \fer{as:M10}} implies formally that, for all $t>0$ and  as $\e\to 0$, $\bar{Z}_\e(t)\to \bar Z(t)$, with
\begin{equation}
\label{eq:Z}
\left\{
\begin{aligned}
\dot{\bar{Z}} (t) &= -  m'(\bar{Z} (t)) ,\\
\bar{Z} (0) & = x_0.  
\end{aligned}
\right.
\end{equation}
In Section \ref{sec:aprho},  we give a quantitative description of this convergence.

\subsection{Main results}
We first  provide a preliminary result guaranteeing  the well-posedness of the Cauchy problem.  
\begin{proposition}
\label{prop:Cauchy}
Let $\ell \in \mathbb N^*$, $\ell \geq 2,$  and assume that $m$ satisfies (H0) and (H3) and $n_0 \in L^1(\mathbb R) \setminus \{0\}$ is nonnegative and has more than $2\ell$ bounded moments. 
Given $\varepsilon >0,$  there exists a unique global solution $n_\e$ to equation \fer{eq_main} with $2\ell$ bounded moments.  This solution is nonnegative and all moments of $n_\e$ of order less than $2\ell$ are well defined and continuous on $[0,\infty).$ 
\end{proposition}
In the case where $m$ is considered bounded, the proposition above is a direct consequence of the Cauchy-Lipschitz Theorem. In the case of unbounded $m$ the result can be obtained via a standard truncation procedure.  We provide the main elements of the proof in Appendix \ref{sec:appendix} in which we make precise the definition of solution.  We show in particular that the $j$-th moment of $n$ is also $C^1$ in time if $j < 2\ell-(p+3).$ 

From now on we assume that $n_{\e}$ is a solution to \eqref{eq_main}  with an arbitrary large number of bounded moments.  We remark that $q_{\e}$ is then a solution to \eqref{eq_rm} with the same regularity/integrability.  We next state our main result.
\begin{theorem}
\label{thm:main}
Assume (H1)--(H4). \\
(i) Let $\delta\in (0,1)$.  There exists a constant $K$ depending on the initial condition,   {$r $ and $m$} such that,  for any $\varepsilon$ sufficiently small, there holds 
\beq
\label{est-main-Was}
\sup_{[0,+\infty)} W_1(q_{\e }(t,\cdot),g_{\varepsilon }(t,\cdot)) \leq  K \left( W_1(q_{\e,0},g_{\varepsilon,0}) + \varepsilon^{1-\delta} \right).
\eeq
with $q_{\e,0}$ and $g_{\varepsilon,0}$ the respective values at time $t=0$ of $q_\e$ and $g_{\varepsilon}.$\\
(ii) {Assume additionally (H5)--(H6) and let $\delta\in (0,1),$ and $\beta\in (1,2)$. Then, there exists a constant $K_\rho$ such that, for any $\varepsilon$ sufficiently small, we have
\beq
\label{as-rho}
| \rho_\e(t)-\rho(t)|\leq K_\rho\e^{1-\delta},\qquad \rho(t)=\frac{r-m(\bar Z(t))}{\kappa},\quad \text{for all $t\in [\e^\beta,+\infty)$}
\eeq
  where $\bar Z$ solves \fer{eq:Z}. Moreover, as $\e\to 0$, $n_\e$ converges in $C\left((0,\infty) ; \mathcal{M}^1(\R) \right)$ to a measure $n$, which is given by 
 $$
 n(t,x)=\rho(t)\partial(x-\bar Z(t)).
 $$}
\end{theorem}
Here we denote $\partial(x-X_0)$ the Dirac mass centered in $X_0.$ We use the (non-standard) symbol $\partial$ in order to avoid confusion with the parameter $\delta$. In this statement, we denote also $W_1$ the wasserstein distance for the $L^1$ distance.
The proof of the theorem above relies on some estimates on the moments of the phenotypic distribution and a contraction property of the Wasserstein distance satisfied by the reproduction operator $\widetilde T_{\varepsilon}$.
 We state below our estimates on the moments of the phenotypic distribution.
 \begin{theorem} \label{thm_ode}
Assume (H1)--(H4) and $\delta \in (0,1).$   For any $\e$ sufficiently small,  there exist large enough positive constants $K_0,K_1$ and $K_2$  depending on the parameters $C_1,r,L$ such that:
\begin{align}
 \label{eq_M4close} M^{c}_{\e,2k_0}(t) &  \leq  K_{2} \varepsilon^{2k_0}, \\
\label{eq_M2close} 
 \left | M_{\e,2}^{c}(t) -  \varepsilon^2  \right| & \leq  K_1\varepsilon^2 \left( \varepsilon^{1-{\delta}}  + \exp(-rt/2\varepsilon^2) \right), \\
\label{eq_M1close} \left| M_{\e,1}(t) - \bar Z_{\e}(t) \right|  &  \leq  K_0 \varepsilon^{1-{\delta}}. 
\end{align}
for any $t \in \mathbb R^+.$
\end{theorem}

  We emphasize that the parameters $K_0,K_1,K_2$ are independent of $\varepsilon,\delta.$ In the course of the proof,  we shall choose first $K_1$ depending on data $C_1$ (see assumption \eqref{eq_initialdata}),  then $K_2$ depending on $K_1,C_1$ and finally  $K_0$ depending on $K_1,L.$ We restrict $\varepsilon$ according to $K_2,K_0$ and $L$.
 The estimates on the moments are obtained via direct computations of the corresponding equations and  using the Taylor expansion  of the mortality rate $m$ around the mean phenotypic trait. The main ingredients are the presence of some dissipation terms,  due to an algebra satisfied by $T_{\varepsilon}$, and the fact that the central moments of the phenotypic distribution have simple approximations when the segregational variance $\e$ is small.  

 One may wonder whether the method suggested in this article could be adapted to study models with an asexual reproduction term.  In Section \ref{sec:comp-asex} we will explain why the dissipation terms are lacking in this latter case and our method cannot be adapted to such models {straightforwardly}.

 \subsection{State of the art}

The study of integro-differential equations involving the infinitesimal model has gained increasing attention recently from the mathematical community.  Several works have used the contraction property of the Wasserstein distance  to study models which involve the infinitesimal operator  \cite{GR:17,GR:22} (see also \cite{PM.GR:15} for a related work on a model of protein exchanges in a cell population). In \cite{GR:17}, G. Raoul studied a model describing the evolutionary dynamics of a population in an environment with continuous spatial structure, under a small selection assumption. He justified the Gaussian assumption and  the so-called Kirkpatrick and Barton model, {\em i.e. } a system of equations describing the dynamics  of the size of the population and its mean phenotypic trait, as   functions of time and space variables   \cite{MK.NB:97}  (see also \cite{SM.GR:13} for an earlier formal derivation of this model). Later in \cite{GR:22}, G. Raoul also studied a model with homogeneous environment, still in a small selection regime, and obtained that the solution of this problem remains close to a Gaussian distribution with fixed variance, while the mean of the distribution varies according  to an ordinary differential equation  similar to \fer{eq:Z}. He also showed that in long time the solution converges to the steady solution of the problem. This work is different from ours in several ways. First,  in \cite{GR:22} the selection is considered to act on the reproduction and not the mortality term. In other words the mortality rate $m$ is supposed to be constant, while a trait dependent reproduction rate $r(\cdot)$ is taken into account in the reproduction term $T$. More importantly, this reproduction rate is assumed to be constant outside of a compact set. Second, \cite{GR:22} does not consider a  competition term as in \fer{eq_main} but considers a renormalized equation closely related to \fer{eq_rm} such that the total population size remains constant equal to $1$.  Finally,  in \cite{GR:22} the segregational variance is assumed to be of order $1$, while in our work it is of order $\e^2$, but the selection term is supposed to be small, that is $r(x)= 1+\e a(x)$, with $a(x)$  compactly supported. One can make a change of variable in our model, that is $\widetilde n_\e(t,x)=\e n_\e (\e^2 t,\e x)$ and $\widetilde \rho_\e(t)=\rho(\e^2 t)$,  to bring the segregational variance equal to $1$, in order to compare our model with the one studied in \cite{GR:22}. Then \fer{eq_main} becomes 
\beq
\label{scaling-G}
  \partial_t  \widetilde n_\e =  T_1  [ \widetilde n_\e]  
-   \left( m(\e \cdot)  + \kappa \widetilde \rho_\e(t) \right)\widetilde n_\e.
\eeq
Comparing $m(\e\cdot)$ and $r(x)=1+\e a(x)$ we realize that these models have different natures. The most important difference in the outcome of these models is that, since $a$ is  compactly supported, the work of \cite{GR:22} only allows for an order $1$ evolution of the mean phenotypic trait of the population, while our model allows for variations of order $1/\e$, after the change of variable leading to \fer{scaling-G}. Our work is indeed in the weak selection regime only in the sense that the evolutionary dynamics are slower than the demographic dynamics but it still allows,   contrary to \cite{GR:22}, important changes of the mean phenotypic trait in long time.

The scaling that we consider is indeed the one studied in \cite{VC.JG.FP:19,FP:23} (see also \cite{LD:20} for the analysis of a model with a spatial structure {and} where the analysis of the infinitesimal operator is partially formal). This scaling is  inspired from previous works on selection-mutation models with asexual reproduction under the assumption of small mutational variance \cite{OD.PJ.SM.BP:05,GB.BP:08,GB.SM.BP:09,AL.SM.BP:10}.  The main ingredient of this approach is a Hopf-Cole transformation that leads to a Hamilton-Jacobi equation with constraint. In the case of the infinitesimal operator,  a similar scaling was suggested in \cite{VC.JG.FP:19} and an asymptotic analysis of a similar model to \fer{eq_main}, without the competition term, was provided in \cite{VC.JG.FP:19} to characterize steady solutions and in \cite{FP:23} to study the time-dependent problem. Although these studies are also based on the Hopf-Cole transformation,  they are very different from the case of asexual reproduction. These analyses rely on a perturbative analysis and do  not involve   Hamilton-Jacobi equations. Our work allows to provide an alternative approach to the one in \cite{FP:23} that relies on a direct study of the dynamics of the moments rather than a Hopf-Cole transformation.  We believe that this approach could be more easily adapted to study more complex models considering spatial or temporal heterogeneities of the environment.  Note also that our assumptions are  different from those of \cite{FP:23}. In \cite{FP:23} a function $M(t,x)$ was defined as follows (rewriting the definition with our notations)
$$
M(t,x)=r+m(x)-m(\overline Z(t))-m'(\overline Z(t))(x-\overline Z(t)).
$$
Some technical assumptions on the growth of $M$ were made in \cite{FP:23} and  it was additionally assumed that
\beq
\label{ass:Patout}
\inf_{(t,x)\in \R_+\times \R} M(t,x) >0.
\eeq
This assumption may be compared to  Assumption (H2). Assumption (H2) together with the first statement of Assumption (H4) can indeed be replaced by the following 
\beq
\label{our-ass}
0< r+ m(x) -m(\overline Z(t)) ,\qquad \text{for all $(t,x)\in \R_+\times \R$}.
\eeq
This assumption is less restrictive, compared to \fer{ass:Patout}, on the choice of initial condition when the mortality rate $m$ is non-convex and it has for instance several local minima. Note for instance that when the function $m$ has two global minima $x_1<x_2$, Assumption \fer{ass:Patout} does not allow for initial conditions which are concentrated around a point  close but to the right of $x_1$, while Assumption \fer{our-ass} allows for initial conditions which are concentrated around a point located in the convexity zones around $x_1$ or $x_2$. For convex growth rates $m$, Assumption \fer{ass:Patout} is less restrictive since it always holds true.   Note however from \fer{as-rho} that to have a positive asymptotic population size $\rho(t)$ one should have  $r-m(\overline Z(t))>0$, which trivially implies that   $r+m(x)-m(\overline Z(t))>0$. Assumption \fer{our-ass} is not hence  restrictive from the modeling point of view.  Finally it is worth mentioning that in \cite{VC.JG.FP:19} the authors proved the existence of steady solutions which concentrate around local minimum points $x_\ast$ which satisfy the following condition
$$
0< r+ m(x) -m(x_\ast) , \qquad \text{for all $x\in \R$}.
$$

More recent works have also studied  the long time behavior of closely related models, but considering  discrete time, without the assumption of weak selection \cite{VC.TL.DP:23,VC.DP.FS:23}.  Other models of adaptive evolution of populations with sexual reproduction and quantitative traits have been studied in \cite{NF.BP:21,BP.MS.CT:22,LD.SM:22}. An asymptotic analysis of a selection-mutation model with an asymmetric reproduction term, considering for instance traits that are mostly inherited from the female, has been provided in \cite{BP.MS.CT:22}. In \cite{LD.SM:22} a model describing the adaptation of quantitative alleles at two loci in a haploid sexually reproducing population has been studied. In \cite{NF.BP:21} a non-expanding transport distance has been introduced in a model with sexual reproduction considering quantitative traits, but with constant birth and death rates.

\subsection{Plan of the paper}
In Section \ref{sec:moments} we provide the estimates on the moments of the phenotypic distribution and prove Theorem \ref{thm_ode}.  In Section \ref{sec:Was} and Section \ref{sec:aprho} we prove respectively Theorem \ref{thm:main}-(i) and (ii). In Section \ref{sec:comp-asex} we provide a comparison between our work and the study of related models with an asexual reproduction term. Finally the main ingredients of the proof of Proposition \ref{prop:Cauchy} are given in Appendix \ref{sec:appendix}.

\medskip

In all computations we use the symbol $C$ for a harmless constant that may vary between lines. We use labelled constants to keep track of the important parameters.

\section{Estimates on the moments of the phenotypic distribution; the proof of Theorem \ref{thm_ode}}
\label{sec:moments}

In the whole section, $q_{\e}$ is a solution to \eqref{eq_rm} with an arbitrary large number of bounded moments so that we may assume that all the necessary moments to our analysis are $C^1$ in time (see Appendix \ref{sec:appendix}).  
We split the approach of {Theorem \ref{thm_ode}} into two steps. Firstly, we compute {\em a priori} ODEs satisfied by moments of $q_{\e}$.  Secondly, we apply a continuation argument to these ODEs to showing that we can find $K_0,K_1,K_2$ depending only on the data of the problem  such that  \eqref{eq_M4close}-\eqref{eq_M2close}-\eqref{eq_M1close} hold globally in time for arbitrary $\delta < 1$ provided $\varepsilon$ is sufficiently small.

\medskip

Before providing our analysis of the moments of the phenotypic distribution,  we first introduce some notations and ingredients of the proof.
We denote $r^{f}[X]$ the (normalized) remainder in the first order Taylor-Lagrange expansion in $X \in \mathbb R$ of a $C^2$ function $f.$ Namely, 
we set:
\[
r^f[X](x) = \int_0^1 (1-{\MH \sigma}) f^{''}(X + {\sigma}(x-X)) {\rm d}{\sigma},
\]
so that
\beq
\label{Taylor}
f(x) = f(X) + (x-X)  f'(X)+ (x-X)^2 r^{f}[X](x).
\eeq
In case of $f=m$, thanks to (H3), we have 
\begin{equation} \label{eq_expansionm}
|r^{m}[X](x)| \leq CA_m ( {1+\,} |X|^{p} + |x-X|^{p}) \quad \forall \, (x,X) \in \mathbb R^2.
\end{equation}
Finally, a key-quantity below is:
\[
I_{\e,m} (t):= \int_{\mathbb R} m(x) q_\e (t,x){\rm d}x.
\]
Using the Taylor expansion $m(x) = m(M_{\e,1}{(t)}) +m'(M_{\e,1}{(t)})(x-M_{\e,1}{(t)})+ r^m[M_{\e,1}{(t)}](x) (x- {M_{\e,1}(t)})^2 $ and
recalling that $M_{\varepsilon,1}{(t)}$ is the first moment of $q_{\e}(t,\cdot),$ we conclude that 
\begin{equation} \label{eq_Im0}
I_{\e,m}(t) = m(M_{\e,1}(t)) + \int_{\mathbb R} (x- {M_{\e,1}(t)})^2 r^{m}[M_{\e,1}(t)](x)q_\e(t,x){\rm d}x.
\end{equation}
The inequality \eqref{eq_expansionm} entails
\begin{equation} \label{eq_Im}
|I_{\e,m}(t) - m(M_{\e,1}(t))| \leq C{A_m} \left( {(1+ |M_{\e,1}{(t)}|^{p})} M_{\e,2}^c {(t)}+ M_{\e,2+p}^{|c|}{(t)} \right).
\end{equation}

\subsection{Time-evolution of the mean}
Firstly, we multiply \eqref{eq_rm} by $x$ and integrate. 
Using that the first moment  of $\Gamma_{\varepsilon}(x)$
vanishes,  we obtain:
\[
\varepsilon^2 \dot{M}_{\e,1} = -   \left(  \int_{\mathbb R} m(x)xq_\e({\MH \cdot},x){\rm d}x  - I_{\e,m} M_{\e,1} \right) .
\]
We compute the right-hand side $RHS$ of this identity by using taylor expansions.  We write $\Psi(x) = xm(x)$ and expand $\Psi$ and $m$ around $X=M_{\e,1}$ with a taylor formula of order $1.$ 
 For the first term,  we note that $\Psi(x) = m(M_{\e,1})M_{\e,1} + \Psi'(M_{\e,1})(x-M_{\e,1}) + r^{\Psi}[M_{\e,1}](x){(x-M_{\e,1})^2}$. Since $M_{\e,1}$ is the first moment of $q_\e$, we obtain 
\[
\int_{\mathbb R} m(x)xq_\e({\MH \cdot,} x){\rm d}x  = M_{\e,1}m(M_{\e,1}) +  \int_{\mathbb R} (x-M_{\e,1})^2 r^{\Psi}[M_{\e,1}](x) q_\e({\MH \cdot,} x){\rm d}x .
\]
For the second term of the $RHS$, we recall \eqref{eq_Im0}
and we have finally:
 \[
 RHS=  -  \int_{\mathbb R} \left[ r^{\Psi}[M_{\e,1}](x) - M_{\e,1} r^{m}[M_{\e,1}](x)  \right] (x-M_{\e,1})^2 q_\e({\MH \cdot,}x){\rm d}x .
 \]
 Noticing that 
 $\Psi''(x) = x m^{''}(x) + 2 m'(x)$ we conclude that:
 \begin{equation} \label{eq_M1}
 \varepsilon^2 \dot{M}_{\e,1}  +  m'(M_{\e,1}) M_{\e,2}^c= -  F_1
 \end{equation}
 where:
 \begin{multline*}
 F_1 =  \int_{\mathbb R} 
 \left[
  \int_0^1  (1-{\MH \sigma}) \left( m^{''}(M_{\e,1} + {\MH \sigma}(x-M_{\e,1})) {\MH \sigma} (x-M_{\e,1})^3
\right.\right.\\ 
 + 2 (m'(M_{\e,1} + {\MH \sigma}(x-M_{\e,1})) - m'(M_{\e,1}))  (x-M_{\e,1})^2 \Bigr) {\rm d}{\MH \sigma} 
 \Bigr]  q_\e({\MH \cdot,}x){\rm d}x   
 \end{multline*}

At this point, we apply (H3) to obtain:
\[
\begin{aligned}
  |m^{''}(M_{\e,1} + {\MH \sigma}(x-M_{\e,1})) | & \leq CA_m \left(  1+ |M_{\e,1}|^p + |x-M_{\e,1}|^p\right),  \\
|(m'(M_{\e,1} + {\MH \sigma}(x-M_{\e,1})) - m'(M_{\e,1}))|  & \leq C A_m |x-M_{\e,1}| \left(  {1+\,} |M_{\e,1}|^p + |x-M_{\e,1}|^p\right) .
\end{aligned}
\]
Introducing this control in the definition of $F_1$ we  infer:
\begin{equation} \label{eq_F1}
|F_1| \leq CA_m \left( (1+|M_{\e,1}|^p) M_{\e,3}^{|c|} + M_{\e,3+p}^{|c|}\right) .
\end{equation}

\subsection{Time-evolution for the second moment}
Concerning the second centered moment, we compute
\[
\begin{aligned}
\varepsilon^2 \dot{M}_{\e,2}^c 
& = \int_{\mathbb R} (x- M_{\e,1})^2 \varepsilon^2 \partial_t q_\e {\MH (\cdot,x)}dx- 2\varepsilon^2 \dot{M}_{\e,1} \int_{\mathbb R} (x-M_{\e,1}) q_\e {\MH (\cdot,x)}dx \\
&= r \int_{\mathbb R} (x-M_{\e,1})^2 (\widetilde T_{\varepsilon}[q_\e]{(\MH x)} -q_\e{\MH(\cdot,x)}) dx
\\
& \quad -  \left( \int_{\mathbb R} m(x)(x-M_{\e,1})^2 q_\e {\MH (\cdot,x)}dx  -  I_{\e,m}\int_{\mathbb R} (x-M_{\e,1})^2 q_\e{\MH (\cdot,x)} dx \right)
 \\
& = r\left(  \int_{\mathbb R} (x-M_{\e,1})^2 \widetilde T_{\varepsilon}[q_\e]{\MH (x)} dx - M_{\e,2}^c  \right)
-  \left(  \int_{\mathbb R} m(x)(x-M_{\e,1})^2 q_\e{\MH (\cdot,x)}  dx- I_{\e,m} \, M_{\e,2}^c \right).
\end{aligned}
\]
We go now into the details of the right-hand side. 
For the first term, we compute:
\begin{multline*}
\int_{\mathbb R} (x-M_{\e,1})^2 \widetilde T_{\varepsilon}[q_\e]{\MH (x)} dx \\
\begin{aligned}
&  = \int_{\mathbb R} \int_{\mathbb R} \int_{\mathbb R} \left( \left( x  - \dfrac{y+y'}{2} \right)  + \left(\dfrac{y+y'}{2} - M_{\e,1} \right)\right)^2 \Gamma_{\varepsilon} \left( x  - \dfrac{y+y'}{2} \right)  q_\e({\MH \cdot,}y) q_\e({\MH \cdot,}y'){\rm d}y {\rm d}y' {\rm d}x \\
 & = \dfrac{ \varepsilon^2}{2} +  \int_{\mathbb R} \int_{\mathbb R}  \left(\dfrac{(y-M_{\e,1})+(y'-M_{\e,1})}{2}\right)^2  q_\e({\MH \cdot,}y) q_\e({\MH \cdot,}y'){\rm d}y {\rm d}y' \\
 & =  \dfrac{\varepsilon^2}{2} +  \dfrac{M_{\e,2}^c}{2} .
 \end{aligned}
\end{multline*}
Eventually,  we obtain: 
\[
\varepsilon^2 \dot{M}_{\e,2}^c = r \left[\dfrac{\varepsilon^2}{2}  - \dfrac{M_{\e,2}^c}{2}\right] -  \left(  \int_{\mathbb R} m(x)(x-M_{\e,1})^2 q_\e{\MH (\cdot,x)}  dx - I_{\e,m} \, M_{\e,2}^c \right)
\]
We proceed with the same expansion trick as above to compute the last term in parenthesis.  We have:
\[
\begin{aligned}
 \int_{\mathbb R} m(x)(x-M_{\e,1})^2 q_\e{\MH (\cdot,x)} dx&  = m(M_{\e,1}) M_{\e,2}^c + m'(M_{\e,1}) M_{\e,3}^c 
\\
 & +  \int_{\mathbb R}\int_0^1 (1-{\MH \sigma}) m^{''}(M_{\e,1}+ {\MH \sigma}(x-M_{\e,1}))(x-M_{\e,1})^4   {\rm d} {\MH \sigma}\, q_\e({\MH \cdot,}x){\rm d}x,
 \end{aligned}
\]
that we combine with \eqref{eq_Im0} to yield:
\begin{multline*}
 \left(  \int_{\mathbb R} m(x)(x-M_{\e,1})^2 q_\e {\MH (\cdot,x)} dx - I_{\e,m}\, M_{\e,2}^c \right)  =  m'(M_{\e,1}) M_{\e,3}^c\\
 + \int_{\mathbb R} \int_0^1
(1-{\MH \sigma}) m^{''}(M_{\e,1}+ {\MH \sigma}(x-M_{\e,1})) ( (x-M_{\e,1})^4  - (x-M_{\e,1})^2 M_{\e,2}^c )  {\rm d} {\MH \sigma} \,  q_\e({\MH \cdot,}x){\rm d}x, 
\end{multline*}
and finally:
\begin{equation} \label{eq_M2}
\varepsilon^2 \dot{M}_{\e,2}^c +  \dfrac{r}{2} M_{\e,2}^c  =  r \dfrac{\varepsilon^2}{2}  +   F_2 ,
\end{equation}
where $F_2$ corresponds to the right-hand side of this latter identity and is controlled thanks to (H3) by:
\begin{multline} \label{eq_F2}
|F_2| \leq C A_m \left[  {(1+|M_{\e,1}|^p)}\Bigl( |M_{\e,2}^c|^2 +  M_{\e,4}^c\Bigr) \right.\\
\left. + {(|M_{\e,1}|+ |M_{\e,1}|^{p+1})} M_{\e,3}^{|c|} +   M^{|c|}_{\e,4+p} + M_{\e,2}^c M^{|c|}_{\e,2+p}   \right].
\end{multline}

\subsection{Time-evolution for higher-order moments}
Finally,  we compute a time-evolution equation for $M_{\e,2k}^c$ with $k$ sufficiently large. Below, we shall apply this computation with $k = k_0$ fixed by \eqref{eq_aprioriM1}.  Similarly to the case of the second order moment, we derive:
\[
\begin{aligned}
\varepsilon^2 \dot{M}_{\e,2k}^c 
& = \int_{\mathbb R} (x-M_{\e,1})^{ 2k}\varepsilon^2  \partial_t q_\e{\MH (\cdot,x)} \,dx- 2k \varepsilon^2 \dot{M}_{\e,1} \int_{\mathbb R} (x-M_{\e,1})^{2k-1} q_\e{\MH (\cdot,x)} \,dx \\
&  =r\left(  \int_{\mathbb R} (x-M_{\e,1})^{2k}  \widetilde T_{\varepsilon}[q_\e]{\MH (x)} dx - M_{\e,2k}^c \right) -   \int_{\mathbb R} m(x) (x-M_{\e,1})^{2k} q_\e{\MH (\cdot,x)} \, dx+ I_{\e,m} M_{\e,2k}^c   \\
& \qquad + 2k  ( F_1 + { m'(M_{\e,1}) }M_{\e,2}^c) M_{\e,2k-1}^c \\
 & \leq r \left[ \int_{\mathbb R} (x-M_{\e,1})^{2k} \widetilde T_{\varepsilon}[q_\e]{\MH (x) dx} - M_{\e,2k}^c \right] 
  +  m( M_{\e,1}) M_{\e,2k}^c + ( I_{\e,m} - m(M_{\e,1})) M_{\e,2k}^c  \\
  &+ 2k   ( F_1 + {m'(M_{\e,1}}) M_{\e,2}^c) M_{\e,2k-1}^c  .
\end{aligned}
\]
Here we have used the fact that ${\MH m(x)(x-M_{\e,1})^{2k} \geq 0}$. Concerning the first term in the bracket multiplied by $r$ we have actually an algebra based on polynomial expansions. To compute this term, we note that   all odd moments of $\Gamma_{\varepsilon}$ vanish.  By a scaling argument,  we have also that the $2l$ moments of $\Gamma_{\varepsilon}$ read $\sigma_l \varepsilon^{2l}$ with a series of constants $(\sigma_l)_{l\in \mathbb N^*}$ independent of $\varepsilon >0$ such that $\sigma_1 = 1/2.$
We now compute   
\begin{multline*}
\int (x-M_{\e,1})^{2k}  {\MH \widetilde T}_{\varepsilon}[q_\e]{\MH (x) dx}  = \\
\begin{aligned}
 &= 
\int_{\mathbb R} \int_{\mathbb R} \int_{\mathbb R} \left( \left( x  - \dfrac{y+y'}{2} \right)  + \left(\dfrac{y+y'}{2} - M_{\e,1} \right)\right)^{2k} \Gamma_{\varepsilon} \left( x  - \dfrac{y+y'}{2} \right)  q_\e({\MH \cdot,}y) q_\e({\MH \cdot,}y'){\rm d}y {\rm d}y' {\rm d}x \\
& =
\sum_{l=0}^{k} \sum_{j=0}^{2l} \sigma_{k-l} \varepsilon^{2(k-l)}
\dfrac{1 }{4^l}\begin{pmatrix} 2k \\ 2l  \end{pmatrix}  \begin{pmatrix} 2l \\ j   \end{pmatrix} 
\int_{\mathbb R} \int_{\mathbb R} 
(y-M_{\e,1})^{2l-j}(y'-M_{\e,1})^{j} q_\e({\MH \cdot,}y) q_\e({\MH \cdot,}y'){\rm d}y {\rm d}y'\\
&= \dfrac{2}{4^{k}} M_{\e,2k}^c 
+  \sum_{l=0}^{k-1} \sum_{j=0}^{2l}  \sigma_{k-l}\varepsilon^{2(k-l)}
\dfrac{1 }{4^l}
\begin{pmatrix} 2k \\ 2l  \end{pmatrix}  
\begin{pmatrix} 2l \\ j   \end{pmatrix} 
 M_{\e,2l-j}^{{c}} M_{\e,j}^{{c}}
+ 
\sum_{j=2}^{2k-2} 
\dfrac{1 }{4^k} \begin{pmatrix} 2k \\ j   \end{pmatrix}  
M_{\e,2k-j}^{{c}} M_{\e,j}^{{c}} ,
\end{aligned} 
\end{multline*}
where we use the convention that 
$M_{\e,0}^c = 1$ and $M_{\e,1}^c=0.$
Eventually, we obtain:
\begin{equation} \label{eq_M4}
\varepsilon^2 \dot{M}_{\e,2k}^c 
+  \left[ r\left(1- \dfrac{2}{4^k} \right) - m(M_{\e,1})   \right]
M_{\e,2k}^c \,\leq \, F_{2k}
\end{equation}
where  
\[
\begin{aligned}
F_{2k} \leq \, & 2k  \left( |F_1| + {C}A_m  {(|M_{\e,1}|+ |M_{\e,1}|^{p+1})} M_{\e,2}^c\right) |M_{\e,2k-1}^c|  && \text{[Time direvative of $\dot{M}_{\e,1}$]}\\
&  + CA_m M_{\e,2k}^c \left( {(1+|M_{\e,1}|^p)} M_{\e,2}^c + M_{\e,2+p}^{|c|}  \right) && 
\text{[Remainder of $I_{\e,m}$]}\\
& 
+  \sum_{l=0}^{k-1} \sum_{j=0}^{2l}  \sigma_{k-l}\varepsilon^{2(k-l)}
\dfrac{1 }{4^l}
\begin{pmatrix} 2k \\ 2l  \end{pmatrix}  
\begin{pmatrix} 2l \\ j   \end{pmatrix} 
M_{\e,2l-j}^{{c}} M_{\e,j}^{{c}} 
\\
&+ 
\sum_{j=2}^{2k-2} 
\dfrac{1 }{4^k} \begin{pmatrix} 2k \\ j   \end{pmatrix}  
M_{\e,2k-j}^{{c}}  M_{\e,j}^{{c}}. 
\end{aligned}
\]

\subsection{The proof of Theorem \ref{thm_ode}}
We are now in position to prove Theorem \ref{thm_ode}. 
We first recall the function $\bar Z_\e$ given by \fer{eq:Ze}. Due to the   local   strict-convexity assumption on $m$, we have $m'({\MH \bar Z_\e}) \geq A_{0} |{\MH \bar Z_\e}|$  as long as $Z_\e \in  (-L,L).$ In particular, thanks to Assumption \eqref{eq_initialdata},  we have $\bar Z_\e(0)\in (-L,L)$. Hence, we have the following decay 
\[
|\bar Z_\e(t)| \leq \exp (-A_{0}t) |\bar Z_\e(0)|.
\] 
We note also here that we have the following classical result. Assume that $Y\geq 0$ satisfies:
\[
\varepsilon^2 \dot{Y} + k Y \leq  F
\]
with $F \in L^{\infty}(0,T).$ Then there holds:
\beq
\label{ODE-in}
|Y(t)| \leq \dfrac{\|F\|_{L^{\infty}(0,T)}}{k}  {+Y(0)e^{\f {-kt}{\e^2}}}, \quad \forall \, t \in [0,T].
\eeq
Combining this property with the estimates obtained {\MH previously},  we will prove Theorem \ref{thm_ode}.  For this,  given $\varepsilon >0$ and   $K_0,K_1,K_2,$ (to be chosen below) we set:
\[
T_{\e} := \sup \{ t \in [0,\infty) \text{ s.t.   \eqref{eq_M4close}--\eqref{eq_M1close} hold on $[0,t]$} \}.
\]
Our aim is to find a $K_0,K_1,K_2$ so that for small $\varepsilon$
we have $T_{\e} = +\infty.$ We show that such a choice is possible with the following proposition:
\begin{proposition} \label{prop_openclose}
Assume (H1)--(H4) are in force. We can find $K_0$, $K_1$ and $K_2$ large enough,   such that for $\delta \in (0,1)$ and arbitrary small $\varepsilon$ there holds:
\begin{itemize}
\item[i)] $T_{\e} >0$ 
\item[ii)] if $T_\e \in (0,\infty),$ we have 
\begin{align}
 \label{eq_M4open}M^{c}_{\e,2k_0}(t) &  <  K_2 \varepsilon^{2k_0} \\
\label{eq_M2open} 
 \left | M_{\e,2}^{c}(t) - \varepsilon^2 \right | & <  K_1\varepsilon^2 \left( \varepsilon^{\MH 1-\delta} + \exp(-rt/2\varepsilon^2) \right) \\
\label{eq_M1open} \left | M_{\e,1}(t) {- \bar Z_\e(t)} \right |  &  <  K_0  \varepsilon^{\MH 1-\delta}. 
\end{align}
on $[0,T_\e].$
\end{itemize}
\end{proposition}
We notice that the inequalities in \eqref{eq_M4close}--\eqref{eq_M1close} are $\leq$ while in \eqref{eq_M4open}--\eqref{eq_M1open} they are $<$. 
 With the above proposition,  we may  conclude by a standard contradiction argument that \eqref{eq_M4close}-\eqref{eq_M2close}-\eqref{eq_M1close} hold on $[0,\infty)$  since $M_{\e,1}$, $M_{\e,2}^c$ and $M_{\e,2k_0}^c$ are time-continuous.

\medskip

\begin{proof}
Thanks to assumption \eqref{eq_initialdata}, item $i)$ only amounts to choosing $K_0,K_1,K_2$ sufficiently large w.r.t. initial data. 
Indeed, by the second part of \eqref{eq_initialdata}, we have $M^{c}_{\e,2k_0}(0) \leq C_1 \varepsilon^{2k_0}$ so that we have \eqref{eq_M4close} on some initial time-interval whenever $K_2 > C_1.$  Then, by interpolation, we have 
\[
|M^{c}_{\e,2}(0) - \varepsilon^2 | \leq (C_1^{\frac 1k_0} +1) \varepsilon^2.
\] 
Since the right-hand side of \eqref{eq_M2close} is larger than $K_1 \varepsilon^2$ intially, we conclude that
\eqref{eq_M2close} holds on some initial time-interval whenever $K_1 >(C_1^{\frac 1k_0} +1).$ Finally, $M_{\e,1}(0) -\bar{Z}_{\e}(0) =0$ 
so that \eqref{eq_M1close} is also true for $t$ small and $K_0$ arbitrary.

\medskip

We focus now on the proof of item $ii)$.
For the proof, we assume that $K_0,K_1,K_2$ are constructed and we explain the restrictions that they must satisfy so that our computations hold true.   We emphasize that we
shall choose first $K_1$ depending on data $C_1$,  then $K_2$ depending on $K_1,C_1$ and finally  $K_0$ depending on $K_1,L.$ We restrict $\varepsilon$ according to $K_2,K_0$ and $L$.
We fix some time $T_\e >0$ such that \eqref{eq_M4close}--\eqref{eq_M1close} hold true on $[0,T_{\e}]$.
We note first that with \eqref{eq_M2close}--\eqref{eq_M1close} we have directly, when $\varepsilon$
is sufficiently small:
\begin{equation} \label{eq_aprioriM2}
|M_{\e,1}(t)| \leq  L, \qquad 
|M_{\e,2}^c(t)| \leq 2 K_1 \varepsilon^2 \quad \text{ on $(0,T_{\e})$} .
\end{equation}
 Fixing from now on that $K_2>2K_1$ the control \eqref{eq_M4close}--\eqref{eq_M2close} also entails that 
\begin{equation} \label{eq_fullscale}
M_{\e,l}^{|c|} \leq   K_2^{\frac{l-2}{{2k_0-2}}} {(2 K_1)^{\frac{2k_0-l}{{2k_0-2}}}} \varepsilon^{l}\leq K_2 \varepsilon^{l} \,, \quad \forall \, l \in \{2,\ldots,2k_0-1\}.  
\end{equation}
We start now with the computation of $M_2^c$ by noting that \eqref{eq_M2} rewrites:
\begin{equation}
\label{ODE-M2}
\left\{
\begin{aligned}
& \dot{M}_{\e,2}^ c  + \dfrac{r}{2\varepsilon^2} M_{\e,2}^c  = \dfrac{r}{2} +  \dfrac{2}{\varepsilon^2} F_2  \\
& {M}_{\e,2}^c (0)  =  \alpha_{\e,0} \varepsilon^2,
\end{aligned}
\right.
\end{equation}
where, $\alpha_{\e,0} \leq C_1$  and 
\[
|F_2| \leq C A_m \left[   ( 1 + |M_{\e,1}|^p )  \Bigl( |M_{\e,2}^c|^2 +  M_{\e,4}^c\Bigr) +  (|M_{\e,1}| + |M_{\e,1}|^{p+1})   M_{\e,3}^{|c|} +   M^{|c|}_{\e,4+p} + M_{\e,2}^c M^{|c|}_{\e,2+p}   \right].
\]
Plugging the controls \eqref{eq_aprioriM2}--\eqref{eq_fullscale} into this identity
we obtain {\MH that, on $(0,T_{\e}),$ we have:}
\[
\begin{aligned}
|F_2| & \leq CA_m  \left[   (1+  L^p)  \Bigl(K_2^2 \varepsilon^4 +  K_2 \varepsilon^4 \Bigr) +   ({ L + L^{p+1})}  K_2 \varepsilon^3 +  K_2 \varepsilon^{4+p} +  K_2^2 \varepsilon^{4+p} \right]\\
& \leq  C A_m (1+L^{p+1}) K_2^2 \varepsilon^3.
\end{aligned}
\]
By direct integration of \fer{ODE-M2} we conclude that:
\[
\begin{aligned}
M_{\e,2}^c(t) & = \alpha_{\e,0}  \varepsilon^2 \exp\left[ -\dfrac{r}{2\varepsilon^2}t \right] + \int_0^t  \exp\left[ \frac{r}{2\varepsilon^2} (\sigma-t) \right]  \left[\dfrac r2 + \dfrac{2 }{\varepsilon^2} F_2(\sigma) \right]{\rm d}{\sigma} \\
& = \varepsilon^2  +  \varepsilon^2  (\alpha_{\e,0} - 1 )  \exp\left[ - \frac{r}{2\varepsilon^2}t\right]+ \dfrac{2 }{\varepsilon^2}\int_0^t  \exp\left[ \frac{r}{2\varepsilon^2} (\sigma-t) \right]  F_2(\sigma) {\rm d}{\sigma}  .
\end{aligned}
\]
We also note that the bound above on $F_2(t)$ leads to
\[
\begin{aligned}
\left | \dfrac{2 }{\varepsilon^2}\int_0^t  \exp\left[ \frac{r}{2\varepsilon^2} (\sigma-t) \right]  F_2(\sigma) {\rm d}{\sigma} \right|  &  \leq 
\dfrac{4}{r} C A_m (1+L^{p+1}) K_2^2 \varepsilon^3 \\
& \leq \frac{K_1}{2} \varepsilon^{\MH 3-\delta},
\end{aligned}
\]
when $\varepsilon$ is sufficiently small. 
We obtain finally that, provided $\varepsilon$ is sufficiently small
\[
|M_{\e,2}^c(t) - \varepsilon^2| \leq  \varepsilon^2  (C_1 + 1 )  \exp\left[ - \frac{r}{2\varepsilon^2}t\right] +  \frac{K_1}{2}   \varepsilon^{\MH 3-\delta},
\]
and we have \eqref{eq_M2open} if $K_1$ is chosen sufficiently large with respect to   $C_1.$

\medskip

We can then rewrite \eqref{eq_M4} thanks to \eqref{eq_aprioriM1} and  \eqref{eq_aprioriM2}:
\begin{equation} \label{eq_M4bis}
\varepsilon^2 \dot{M}_{\e,2k_0}^c + \eta  M_{\e,2k_0}^c \leq F_{2k_0},
\end{equation}
where we recall that:
\[\begin{aligned}
F_{2k_0} \leq  & 2k_0  \left( |F_1| + {C}A_m  (|M_{\e,1}| +   |M_{\e,1}|^{p+1})  M_{\e,2}^c\right) |M_{\e,2k_0-1}^c|  \\
&  + CA_m M_{\e,2k_0}^c \left(  (1 + |M_{\e,1}|^p)  M_{\e,2}^c + M_{\e,2+p}^{|c|}   + M_{\e,1}^{2} + M_{\e,1}^{p+2}\right) \\
& 
+  \sum_{l=0}^{k_0-1} \sum_{j=0}^{2l}  \sigma_{k_0-l}\varepsilon^{2(k_0	-l)}
\dfrac{1 }{4^l}
\begin{pmatrix} 2k_0 \\ 2l  \end{pmatrix}  
\begin{pmatrix} 2l \\ j   \end{pmatrix} 
 M_{\e,2l-j}^{{c}} M_{\e,j}^{{c}}
+ 
\sum_{j=2}^{2k_0-2} 
\dfrac{1 }{4^{k_0}}  \begin{pmatrix} 2k_0 \\ j   \end{pmatrix}  
M_{\e,2k_0-j}^{{c}}  M_{\e,j}^{{c}} ,
\end{aligned}
\]
with 
\[
|F_1| \leq CA_m (  (1 + |M_{\e,1}|^{p} ) M_{\e,3}^{|c|} + M_{\e,3+p}^{|c|}).
\]
Introducing again {Assumption \fer{eq_initialdata}}, \eqref{eq_aprioriM2} and \eqref{eq_fullscale} we obtain that:
\begin{multline*}
2k_0  \left( |F_1| + {C}A_m  (|M_{\e,1}| + |M_{\e,1}|^{p+1})  M_{\e,2}^c\right) |M_{\e,2k_0-1}^c|  \\
+   CA_m M_{2k_0}^c \left( ( 1+ |M_{\e,1}|^p)   M_{\e,2}^c + M_{\e,2+p}^{|c|}  \right) 
   \leq C A_m (1+L^{p+1}) K_2^2  \varepsilon^{2k_0+1}.  
\end{multline*}
{\MH Choosing $\varepsilon$ sufficiently small, we can then make the left-hand side of this inequality smaller than $\alpha K_2\varepsilon^{2k_0}$ with $\alpha$ arbitrary small.}
Concerning the last line of $F_{2k_0}$ we use the first interpolation inequalities in \eqref{eq_fullscale}  (since this quantity only involves moments of order less than $2k_0-1$) to yield (we recall that we keep only larger powers of $K_1,K_2$ since they are assumed larger than 1): 
\begin{multline*}
\sum_{j=2}^{2k_0-2} 
\dfrac{1 }{4^{k_0}}  \begin{pmatrix} 2k_0 \\ j   \end{pmatrix}  
M_{\e,2k_0-j}^{{c}} M_{\e,j}^{{c}}
+ 
\sum_{l=0}^{k_0-1} \sum_{j=0}^{2l}  \sigma_{k_0-l}\varepsilon^{2(k_0	-l)}
\dfrac{1 }{4^l}
\begin{pmatrix} 2k_0 \\ 2l  \end{pmatrix}  
\begin{pmatrix} 2l \\ j   \end{pmatrix} 
M_{\e,2l-j}^{{c}} M_j^{\e,{c}}\\
\leq 
C_{k_0} K_1^{\frac{2k_0}{k_0-1}} K_2^{{\frac{k_0-2}{k_0-1}}}  \varepsilon^{2k_0}
\end{multline*}
We point out that the constant $C_{k_0}$ depends on $k_0$ also through the explicit quantities $\sigma_1,\ldots,\sigma_{k_0}.$  {\MH Finally,} we can   choose  $K_2$  large with respect to  $K_1$  to conclude 
that:
\[
F_{2k_0} \leq  \dfrac{\eta}{2} K_2 \varepsilon^{2k_0} \text{ \MH on $(0,T_{\varepsilon})$}.
\]
Integrating \eqref{eq_M4bis} and using Assumption \fer{eq_initialdata} entails that:
\[
M_{\e,2k}^c{(\MH t)} \leq C_1 \varepsilon^{2k_0} + \dfrac{K_2}{2} \varepsilon^{2k_0},
\]
and we have \eqref{eq_M4open} when $K_2$ is again sufficiently large with respect to $C_1$.

\medskip

We end up with the computation of $M_{\e,1}.$ We can rewrite 
\eqref{eq_M1} as follows
\[
 \dot{M}_{\e,1}  +   m'(M_{\e,1})  =  \tilde{F}_1(t)
\]
where thanks to  \fer{eq_M4close}, \eqref{eq_M2close} and \eqref{eq_aprioriM2}:
\[
\begin{aligned}
|\tilde{F}_1(t)|  & \leq \dfrac{|F_1|}{\varepsilon^2} +  \dfrac{1}{\varepsilon^2} |  m'(M_{\e,1})| \,|M_{\e,2}^c - \varepsilon^2| \leq  
{\MH CA_m (1+L^{p+1})K_1}{\left( \varepsilon^{\MH 1-\delta}  + \exp(-rt/2\varepsilon^2) \right)}.
\end{aligned}
\]
Hence $N_1 = M_{\e,1} - \bar Z_\e$ satisfies 
\[
\dot{N}_1 +   m''(\bar Z_\e) N_1 = \tilde{F}_1 -   (m'(M_{\e,1}) -m'(\bar Z_\e) -  m''(\bar Z_\e) N_1 ), \qquad N_1(0) = 0.
\]
Moreover,  we can write $m'(M_{\e,1}) -m'(\bar Z_{\e}) = m''(Z_{\theta})N_1$ for some $Z_{\theta}$ between $\min(M_{\e,1},\bar Z_{\e})$ and $\max(M_{\e,1},\bar Z_{\e})$ and use again a finite difference theorem to compute $m''(Z_{\theta}) - m''(\bar Z_{\e})$. This entails thanks to \fer{eq_M1close}:
\[
|   m'(M_{\e,1}) -m'(\bar Z_\e) -  m''(\bar Z_\e) N_1  |
\leq C_m  K_0 \varepsilon^{\MH 1-\delta} |N_1|. 
\]
with a constant $C_m$ depending on $m^{'''}.$
Consequently, we have by a standard energy estimate
\[
\dfrac{\textrm{d}}{\textrm{d}t}  \left[ \dfrac{|N_1|^2}{2}\right] + A_0 |N_1|^2  \leq  (\tilde{F}_1 + C_m K_0 \varepsilon^{1-}|N_1|) N_1 .
\]
Eventually, we obain that, for $\varepsilon$ sufficiently small: 
\[
\begin{aligned}
\dfrac{|N_1(t)|^2}{2}&  \leq  {\MH \dfrac{C}{A_0}}\int_0^{t}  \exp(-A_0(t-s)) |\tilde{F_1}(s)|^2 {\rm d}s \\
&  \leq  {\MH \dfrac{C A_m^2 (1+L^{p+1})^2 K_1^2 }{A_0}} \left(  \int_0^t \exp(-A_0(t-s)) \e^{2(1-\delta)} {\rm d}s + \int_0^t \exp(-rs/{\varepsilon^2}) {\rm d}s\right)\\
& \leq  {\MH \dfrac{C A_m^2 (1+L^{p+1})^2 K_1^2 }{A_0}}  \left( \dfrac{\e^{2(1-\delta)}}{A_0} + \dfrac{\e^2}{r} \right)
\end{aligned}
\]
We obtain \eqref{eq_M1close} provided $K_0$ is chosen sufficiently large with respect to $K_1,$ {\MH $A_0$} and $L$.
This concludes the proof.
{\qed}
\end{proof}

\section{ Wasserstein estimates;  The proof of Theorem \ref{thm:main}-(i).}
\label{sec:Was}

In this section, we assume that $q_{\e}$ is a global solution to \eqref{eq_rm} with sufficiently large bounded moments. We assume that 
the conclusions of {Theorem \ref{thm_ode}} hold true and we recall that notations $g_{\varepsilon},\bar{Z}_{\e}$ are defined in the introduction. Again, we use throughout the section the symbol $C$ for a harmless constant that may vary between lines. We reserve the symbol $K$ for a constant that depend only on $L$, $K_0$, $K_1$ and $K_2$. This constant may also vary between lines.

\medskip

We first remark that $g_{\varepsilon}$ solves
\begin{equation} \label{eq_evolg}
\varepsilon^2 \partial_t g_{\varepsilon}(t,x)
= r(\tilde{T}[g_{\varepsilon}](t,x) - g_{\varepsilon}(t,x)) -   (x-\bar{Z}_{\e}(t)) m'(\bar{Z}_{\e}(t)) g_{\varepsilon}(t,x) .
\end{equation}
We will use below that the last term writes:
\[
\begin{aligned}
 (x-\bar{Z}_{\e})m'(\bar{Z}_{\e}) 
& =  m(x) - m(\bar{Z}_{\e} )  - r^{m}[{\bar{Z}_{\e}}](x)(x-\bar{Z}_{\e})^2
\end{aligned}
\]
where:
\[
| r^{m}[\bar Z_{\e}](x)| \leq C A_m \left(  1  + |\bar{Z}_{\e}|^{p} + |x-\bar{Z}_{\e}|^{p}\right).
\]
We also define  
\begin{equation} \label{eq_gbare}
\bar{g}_{\varepsilon}(t,x) =  \dfrac{1}{\sqrt{2\pi}\varepsilon} \exp \left(-  \dfrac{(x-{M_{\e,1}}(t))^2}{2\varepsilon^2}\right)
\end{equation}
with ${M_{\e,1}}$ given by \fer{moments}.  

\medskip

We propose now to derive some duality estimates for the difference $(q_{\e} - g_{\varepsilon})$ by computing:
\[
I_{\phi} (t)= \int_{\mathbb R} \phi (x)(q_\e(t,x)-g_{\varepsilon}(t,x))dx,
\]
for arbitrary lipschitz bounded  (possiblity smooth) test-function such that $\phi(0)=0$.  Multiplying \eqref{eq_rm} and \eqref{eq_evolg} by $\phi $ 
and taking the difference, we get 
\begin{equation} \label{eq_Iphi}
\varepsilon^2 \dot{I}_{\phi} = r  T_r - r I_{\phi}   -   T_{s},
\end{equation}
where 
\[
\begin{aligned}
T_r &= \int_{\mathbb R} \phi (\tilde{T}_{\varepsilon}[q_\e] - \tilde{T}_{\varepsilon}[g_{\varepsilon}]) ,\\
T_s & = \left( \int_{\mathbb R} mq_\e \phi  - \int_{\mathbb R} mg_{\varepsilon} \phi \right)  -
\left[  \left( \int_{\mathbb R} mq_\e \right) \left( \int_{\mathbb R} q_\e \phi \right) - m({\MH \bar{Z}_{\e}})
 \left( \int_{\mathbb R} g_{\varepsilon} \phi \right)\right]\\
 & \quad  {\MH -} \int_{\mathbb R} (x-\bar{Z}_{\e})^2 r^{m}[\bar{Z}_{\e}]{\MH(x)} g_{\varepsilon}{\MH (\cdot,x)} \phi{\MH (x)}dx \\
 & = T_{s}^{(a)} - T_{s}^{(b)} {+}  T_{s}^{(c)}.
 \end{aligned}
\]
We proceed by controlling now $T_r$ and $T_s.$ Concerning  $T_r$, we write
\[
T_r = \int_{\mathbb R} \phi (\tilde{T}_{\varepsilon}[q_\e] - \tilde{T}_{\varepsilon}[\bar{g}_{\varepsilon}]) + \int_{\mathbb R} \phi ( \tilde{T}_{\varepsilon}[\bar{g}_{\varepsilon}] -\tilde{T}_{\varepsilon}[{g}_{\varepsilon}]  ) .
\]
For the first term in $T_r$ we use Tanaka's estimate \cite[Appendix 4.2]{GR:17} and obtain
\[
\begin{array}{rl}
\left|\displaystyle \int_{\mathbb R} \phi (\tilde{T}_{\varepsilon}[q_\e] - \tilde{T}_{\varepsilon}[\bar{g}_{\varepsilon}]) \right|
&\leq \|\phi'\|_{L^{\infty}}  W_1(\tilde{T}_{\varepsilon}[q_\e],\tilde{T}_{\varepsilon}[\bar{g}_\e]) \leq \|\phi'\|_{L^{\infty}}  W_2(\tilde{T}_{\varepsilon}[q_\e],\tilde{T}_{\varepsilon}[\bar{g}_\e])
\\
& \leq \dfrac{\|\phi'\|_{L^{\infty}}}{\sqrt{2}} W_2(q_\e,\bar{g}_{\varepsilon})  \leq \dfrac{\|\phi'\|_{L^{\infty}}}{\sqrt{2}} \left( W_2(q_\e,{g}_{\varepsilon})+ W_2(g_{\varepsilon},\bar{g}_{\varepsilon})\right),
\end{array}
\]
where:
\[
\begin{aligned}
W_2(g_{\varepsilon},\bar{g}_{\varepsilon})&  \leq \left( \int_{\mathbb R}\int_{\mathbb R} |x_1-x_2|^2 g_{\varepsilon}(x_1)\bar{g}_{\varepsilon}(x_2) dx_1dx_2 \right)^{\frac 12} \\
&  \leq C \left( \int_{\mathbb R}\int_{\mathbb R} |(x_1  -\bar{Z}_{\e}) -(x_2-M_{\e,1})|^2 g_{\varepsilon}({\MH \cdot,}x_1)\bar{g}_{\varepsilon}({\MH \cdot,}x_2)  dx_1dx_2\right)^{\frac 12} 
+ C|M_{\e,1}-\bar{Z}_{\e}| \\
&  \leq C \left( 2\int_{\mathbb R}\int_{\mathbb R} (|x_1-\bar{Z}_{\e}|^2 + |x_2- M_{\e,1}|^2) g_{\varepsilon}({\MH \cdot,}x_1)\bar{g}_{\varepsilon}({\MH \cdot,}x_2)  dx_1dx_2\right)^{\frac 12} +  C|M_{\e,1}-\bar{Z}_{\e}| \\
&  \leq C \varepsilon + C |M_{\e,1}-\bar{Z}_{\e}| . 
\end{aligned}
\]
Furthermore,  for any $\pi$ in the transference plane between $q_\e$ and $g_{\varepsilon},$ we have  
 \begin{multline*}
 \int_{\mathbb R}\int_{\mathbb R} |x_1-x_2|^2 d\pi(x_1,x_2) \\
\begin{aligned}
&  \leq \left( \int_{\mathbb R}\int_{\mathbb R} |x_1-x_2| d\pi(x_1,x_2)\right)^{\frac 23}\left( \int_{\mathbb R}\int_{\mathbb R} |x_1-x_2|^4 d\pi(x_1,x_2)\right)^{\frac 13}\\
& \leq {C}  \left( \int_{\mathbb R}\int_{\mathbb R} |x_1-x_2| d\pi(x_1,x_2)\right)^{\frac 23}
\left(\left[\int_{\mathbb R}\int_{\mathbb R} ( |x_1 - M_{\e,1}|^4 +|x_2-\bar{Z}_{\e}|^4) d\pi(x_1,x_2)  \right]^{\frac 13} \right. \\ 
& \left. \qquad + |M_{\e,1}- \bar{Z}_{\e}|^{\frac 43} \right)\\
& \leq  {c_0^2 }  \left( \int_{\mathbb R}\int_{\mathbb R}|x_1-x_2| d\pi(x_1,x_2)\right)^2  +  C \left( \dfrac{M_{\e,4}^c + \varepsilon^4 + |M_{\e,1}-\bar{Z}_{\e}|^{4}}{{c_0^2 }}\right)^{\f 12}
\end{aligned} 
\end{multline*}
where $c_0$ can be chosen arbitrarily small.
With a classical inf argument, we obtain 
\[
 {W_2(q_\e,g_{\varepsilon})^2} \leq  {c_0^2 } W_1(q_\e,g_{\varepsilon})^2 + C  \left( \dfrac{ M_{\e,4}^c + \varepsilon^4 + |M_{\e,1}-\bar{Z}_{\e}|^{4}}{{c_0^2 }}\right)^{\f12}.
\]
It follows that
\[
W_2(q_\e,g_{\varepsilon})  \leq   {c_0 } W_1(q_\e,g_{\varepsilon})  + C \left( \dfrac{ M_{\e,4}^c + \varepsilon^4 + |M_{\e,1}-\bar{Z}_{\e}|^{4}}{{c_0^2 }}\right)^{\f14}.
\]
and 
\[
\left|\int \phi (\tilde{T}_{\varepsilon}[q_\e] - \tilde{T}_{\varepsilon}[\bar{g}_{\varepsilon}]) \right| 
\leq \|\phi'\|_{L^{\infty}} \left(   c_0 W_1(q_\e,g_{\varepsilon}) 
+ C \left( \dfrac{ \varepsilon + |M_{\e,1} - \bar{Z}_{\e}| + (M_{\e,4}^{c})^{\f14} }{c_0^{\f12}}\right)  \right)
\]
To control the second term in $T_r,$ we can simply use
$$
\tilde{T}_{\varepsilon}[\bar{g}_{\varepsilon}]=\bar{g}_{\varepsilon},\qquad \tilde{T}_{\varepsilon}[{g}_{\varepsilon}]={g}_{\varepsilon}.
$$
We then obtain with a similar decomposition of $W_{1}(\bar{g}_{\e},g_{\e})$ as previously:
\[
\left|\int_{\mathbb R} \phi (\tilde{T}_{\varepsilon}[\bar{g}_{\varepsilon}] -\tilde{T}_{\varepsilon}[{g}_{\varepsilon}]  )   \right|=  
\left|\int_{\mathbb R} \phi (\bar{g}_{\varepsilon} -{g}_{\varepsilon}  )   \right|
\leq \|\phi'\|_{L^{\infty}} W_1(\bar{g}_{\varepsilon} ,{g}_{\varepsilon}) \leq C(\e+|M_{\e,1}-\bar Z_{\e}|)
\]
and
\[
|T_r| \leq \|\phi'\|_{L^{\infty}} \left[  c_0 W_1(q_\e,g_{\varepsilon}) +C \left( \dfrac{ \varepsilon + |M_{\e,1} - \bar{Z}_{\e}| + (M_{\e,4}^{c})^{\f14} }{c_0^{\f12}}    \right) \right].
\]
Introducing the result of {Theorem \ref{thm_ode}} we conclude that:
\begin{equation} \label{est_Tr}
|T_r| \leq \|\phi'\| \left[c_0 W_1(q_\e,g_{\varepsilon})+ \dfrac{K}{c_0^{{1/2}}} \varepsilon^{\MH 1-\delta} \right]
\end{equation}
where $K$ is a constant that depend only on $L$, $K_0$, $K_1$ and $K_2$.

\medskip

We turn to estimating $T_s.$ We use below without mention the controls induced by {Theorem \ref{thm_ode}}.  Concerning $T_s^{(c)},$ we apply (H3) to bound $m^{''}$ and we
have:
\[
\begin{aligned}
|T_s^{(c)}|\leq&   CA_m \int_{\mathbb R} (  1+  |\bar{Z}_{\e}|^{p} + |x-\bar{Z}_{\e}|^{p})  |x-\bar{Z}_{\e}|^2 (|\phi(\bar{Z}_{\e})| + |\phi(x)-\phi(\bar{Z}_{\e})|) g_{\varepsilon}({\MH \cdot,}x) dx\\
\leq&   CA_m\|\phi'\|_{L^{\infty}} \left( (1 + |\bar{Z}_{\e}|^p) \int_{\mathbb R} |x-\bar{Z}_{\e}|^3 g_{\varepsilon}(x) dx+{ (1 + |\bar{Z}_{\e}| ^{p}) |\bar{Z}_{\e}|} \int_{\mathbb R} |x-\bar{Z}_{\e}|^2 g_{\varepsilon}({\MH \cdot,}x)dx\right.
\\&
\left. + |\bar{Z}_{\e}| \int_{\mathbb R} |x-\bar{Z}_{\e}|^{p+2} g_{\varepsilon}({\MH \cdot,}x)dx + \int_{\mathbb R} |x-\bar{Z}_{\e}|^{p+3} g_{\varepsilon}({\MH \cdot,}x) dx\right) \\
\leq & CA_m  \|\phi'\|_{L^{\infty}} (1+|\bar{Z}_{\e}|^{p+1})  \varepsilon^2  .
\end{aligned}
\] 
Eventually, we conclude that:
\begin{equation} \label{eq_Tsc}
|T_s^{(c)}|\leq  K \|\phi'\|_{L^{\infty}}  \varepsilon^2.
\end{equation}
Concerning $T_s^{(b)},$ we have 
\[
\left| T_s^{(b)} -   m(\bar{Z}_{\e}) \int_{\mathbb R} (q_\e-g_{\varepsilon}) \phi   \right|  \leq   \left| \int_{\mathbb R} (m- m(\bar{Z}_{\e}))q_\e  \right| \left| \int_{\mathbb R} q_\e\phi \right|  
\]
where, applying Taylor expansion of $m$ in $\bar{Z}_{\e}:$
\[
\begin{aligned}
\left| \int_{\mathbb R} (m-m(\bar{Z}_{\e})) q_\e  \right|&  \leq \int |r^{m}[\bar{Z}_{\e}](x)| (x-\bar{Z}_{\e})^2 q_\e ({\MH \cdot,}x){\rm d}x+ |m'(\bar{Z}_{\e})| |M_{\e,1}-\bar Z_{\e}|
\\
&
\leq C A_m \left( \int_{\mathbb R} \left({ (1 +  |\bar{Z}_{\e}|^{p} )}|x-\bar{Z}_{\e}|^2 + |x-\bar{Z}_{\e}|^{p+2} \right) q_\e({\MH \cdot,}x) {\rm d}x  \right) {+K|M_{\e,1}-\bar Z_{\e}|}\\
& \leq  CA_m \left({ (1 +  |\bar{Z}_{\e}|^{p}) (M_{\e,2}^{c} + |M_{\e,1} - \bar{Z}_{\e}|^2)} + M_{\e,2+p}^{|c|} + |M_{\e,1} - \bar{Z}_{\e}|^{2+p} \right)\\
& \qquad {+K |M_{\e,1}-\bar Z_{\e}|},
\end{aligned}
\]
and, with a standard Jensen inequality:
\[
\left| \int_{\mathbb R} q_\e \phi  \right| \leq \|\phi'\|_{L^{\infty}}\left(  \int_{\mathbb R} q_\e({\MH \cdot,}x) |x|^2 dx\right)^{\frac 12} \leq  \|\phi'\|_{L^{\infty}} \left|M_{\e,2}^c + M_{\e,1}^{2} \right|^{\frac 12} \leq K \|\phi'\|_{L^{\infty}}.
\]
This entails :
\begin{equation} \label{eq_Tsb}
\left| T_s^{(b)} -   m(\bar{Z}_{\e}) I_{\phi} \right|  \leq K\e^{\MH 1-\delta} \|\phi'\|_{L^{\infty}}.
\end{equation}
Finally we compute the third term, that is 
\[
T_s^{(a)}   = \int_{\mathbb R} m (q_\e-g_{\varepsilon}) \phi  = \int_{\mathbb R} m q_\e \phi - \int_{\mathbb R} m g_{\varepsilon} \phi  .
\]
We use again Taylor expansion 
\begin{multline*}
\left| \int_{\mathbb R} m q_\e \phi  - m(M_{\e,1}) \int_{\mathbb R} q_\e \phi  \right| \\ 
\begin{aligned}
& = \left| \int_{\mathbb R}  r^{m}[{M}_{\e,1}](x)(x-M_{\e,1})^2 q_\e({\MH\cdot,}x) \phi(x) {\rm d}x
 + m'(M_{\e,1}) \int_{\mathbb R} (x-M_{\e,1})q_\e({\MH \cdot,}x) \phi(x) {\rm d}x \right| \\
  & \leq   K \varepsilon  \|\phi'\|_{L^{\infty}}.
\end{aligned}
\end{multline*}
Proceeding similarly with  $g_{\varepsilon}$ centered in ${\MH \bar{Z}_{\e}}$ we end up with 
\begin{equation} \label{eq_Tsa}
\left| T_s^{(a)} -    (m(M_{\e,1}) - m(\bar{Z}_{\e})) \int q_\e \phi - m(\bar{Z_{\e}}) I_{\phi} \right|   \leq  K \varepsilon \|\phi'\|_{L^{\infty}} .
\end{equation}
Combining \eqref{eq_Tsc}-\eqref{eq_Tsb}-\eqref{eq_Tsa},  we obtain  
\[
|T_s| \leq K \varepsilon^{\MH 1-\delta} \|\phi'\|_{L^{\infty}}.
\]
Eventually, we conclude that 
\[
\varepsilon^{2} \dot{I}_{\phi} {+}  r I_{\phi} = F
\]
where 
\[
|F| \leq \|\phi'\|_{\MH L^{\infty}} \left[rc_0 W_1(q_\e,g_{\varepsilon})+ \dfrac{K}{{c_0^{1/2}}} \varepsilon^{\MH 1-\delta} \right]. 
\]
Applying the property \fer{ODE-in},  we conclude that 
\[
\begin{aligned}
\left| {I}_{\phi}(t) \right| & \leq  I_{\phi}(0) +  \dfrac{1}{r}  \sup_{(0,t)} |F(t)| \\
& \leq \|\phi'\|_{L^{\infty}} \left( W_1(q_{\e,0},g_{\varepsilon,0}) + {c_0} \sup_{(0,t)} W_1(q_\e(t,\cdot),g_{\varepsilon }(t,\cdot)) + \dfrac{K}{{{r}c_0^{1/2}}} \varepsilon^{\MH 1-\delta} \right).
\end{aligned}
\]
Applying a sup argument on $\phi \in C^{1}_b(\mathbb R)$ such that  $\|\phi'\|_{L^{\infty}} \leq 1$ and then on $t \in [0,T]$ (for fixed $T < \infty$) we obtain that 
\[
\sup_{[0,T]} W_1(q_\e(t,\cdot),g_{\varepsilon }(t,\cdot))  \leq  W_1(q_{\e,0},g_{\varepsilon,0})+ c_0 \sup_{(0,T)} W_1(q_\e(t,\cdot),g_{\varepsilon }(t,\cdot)) + \dfrac{K}{{{r}c_0^{1/2}}} \varepsilon^{\MH 1-\delta}
\]
and thus, since $c_0$ can be chosen arbitrary small, we obtain
\[
\sup_{[0,T]} W_1(q_\e(t,\cdot),g_{\varepsilon }(t,\cdot))  \leq  K \left( W_1(q_{\e,0},g_{\varepsilon,0}) + \varepsilon^{\MH 1-\delta} \right).
\]
This ends the proof {of (i)}.

\section{Approximation of the population size $\rho_\e$; the proof of Theorem \ref{thm:main}-(ii)}
\label{sec:aprho}
{In this section we provide the proof of Theorem \ref{thm:main}-(ii), that is the approximation of the total population size $\rho_\e$ and the convergence of $n_\e$ as $\e\to 0$.}

\medskip

{(a) {\bf The approximation of $\rho_\e$.}} We integrate the first line of  \fer{eq_main} with respect to $x$ to obtain
\beq
\label{eq-rho}
\e^2 \f{d}{dt} \rho_\e=\rho_\e \left( r-\int_\R m(y)q_\e(t,y)dy-\kappa \rho_\e \right).
\eeq
We use this equation to prove the second statement of Theorem \ref{thm:main} in two steps. 

{\bf Step 1.} We first prove that 
\beq
\label{mq-mz}
\int_\R m(y)q_\e(t,y)dy=m \big(\overline Z(t)\big)+O(\e^{\MH 1-\delta}).
\eeq
We  use the Taylor expansion \fer{Taylor} for $f=m$ at $z=M_{\e,1}$   to find
$$
\begin{aligned}
\int_\R m(y)q_\e(t,y)dy=& m(M_{\e,1}{\MH (t)})\int_R q_\e(t,y)dy+m'(M_{\e,1}{\MH (t)})\int_\R (y-M_{\e,1}{\MH (t)}) q_\e(t,y)dy\\
&+\int_\R (y-M_{\e,1}{\MH (t)})^2r^m[M_{\e,1}{\MH (t)}](y)q_\e(t,y)dy\\
=& m(M_{\e,1}{\MH (t)})+\int_\R (y-M_{\e,1}{\MH (t)})^2r^m[M_{\e,1}{\MH(t)}](y)q_\e(t,y)dy.
\end{aligned}
$$
Combining this equality with  \fer{eq_expansionm} we obtain
$$
\big|\int_\R m(y)q_\e({\MH \cdot,}y)dy-m(M_{\e,1})\big|\leq C\int_\R(y-M_{\e,1})^2 \big(1+|M_{\e,1}|^p+|y-M_{\e,1}|^p \big) q_\e(t,y)dy.
$$
Note from \fer{eq_initialdata}, \fer{eq:Ze} and \fer{eq_M1close} that, for $\e$ small enough,
$$
|M_{\e,1}(t)|\leq L, \qquad \text{for all $t\in \R^+$}.
$$
We deduce that
$$
\begin{aligned}
\big|\int_\R m(y)q_\e(t,y)dy-m(M_{\e,1}(t))\big| \leq & \, C(1+L^p)\int_\R(y-M_{\e,1}(t))^2   q_\e(t,y)dy\\
&+\int_\R(y-M_{\e,1}(t))^{2+p}   q_\e(t,y)dy.
\end{aligned}
$$
We next use \fer{eq_M2close} and \fer{eq_M4close} to obtain that
$$
\big|\int_\R m(y)q_\e(t,y)dy-m(M_{\e,1}(t))\big| \leq C\e^2 ,
$$
for $\e$ small enough and up to changing the constant $C$.
In order to prove \fer{mq-mz} we next compare $M_{\e,1}(t)$ with $\overline Z(t)$. To this end we first recall from   \fer{eq_M1close} that 
$$
|M_{\e,1}(t)-\overline Z_\e(t)|\leq K_0\e^{\MH 1-\delta}.
$$
Next we subtract \fer{eq:Z} from \fer{eq:Ze} to find that
$$
\f{d}{dt} \big(\overline Z_\e(t)- \overline Z(t)\big)=-m'(\overline Z_\e(t))+m'(\overline Z (t))=-m''(\theta(t))  (\overline Z_\e(t)- \overline Z(t)),
$$
with
$$
\theta(t)\in [\min  (\overline Z_\e(t), \overline Z(t)), \max (\overline Z_\e(t), \overline Z(t))]\subset [-L,L].
$$
From this equality together with Assumption (H1) we deduce that
$$
|\overline Z_\e(t)- \overline Z(t)| \leq e^{-A_0 t}|\overline Z_\e(0)- \overline Z(0)|\leq C\e e^{-A_0 t}.
$$
Combining the inequalities above we obtain that 
$$
|m(M_{\e,1}(t))-m(\overline Z (t))|=O(\e^{\MH 1-\delta}), 
$$
and hence 
$$
\sup_{[0,T]}\big|\int_\R m(y)q_\e(t,y)dy-m(\overline Z (t))\big| =O(\e^{\MH 1-\delta}).
$$

{\bf Step 2.} We next prove \fer{as-rho}. To this end, we first notice, thanks to \fer{eq-rho} and \fer{mq-mz}, that we have
$$
\e^2 \f{d}{dt} \rho_\e=\rho_\e \big( r-m(\overline Z)-\kappa \rho_\e +O(\e^{\MH 1-\delta}) \big).
$$
We then define
$$
J_\e:=\f {1} {\rho_\e}.
$$
Replacing this in the equation on $\rho_\e$ we obtain the following linear equation
$$
\e^2 \f{d}{dt} J_\e= \kappa - J_\e \big( r-m(\overline Z)  +O(\e^{\MH 1-\delta}) \big).
$$
We solve this equation to obtain 
\beq
\label{eq:Je}
J_\e(t) = J_\e(0)e^{-\f{1}{\e^2}\int_0^t \big( r-m(\overline Z(s))  +O(\e^{\MH 1-\delta}) \big) ds}+\f{\kappa}{\e^2}\int_0^t e^{-\f{1}{\e^2}\int_\tau^t \big( r-m(\overline Z(s))  +O(\e^{\MH 1-\delta}) \big) ds}d\tau.
\eeq
We next use \fer{as:M10},  \fer{eq:Ze} and {(H2)} to obtain that there exists a positive constant $a_0$ such that, for $\e$ small enough,
$$
r-m(\overline Z(s))  +O(\e^{\MH 1-\delta})\geq a_0, \qquad \text{for all $s\in \R^+$}.
$$
We deduce that
$$
J_\e(0)e^{-\f{1}{\e^2}\int_0^t \big( r-m(\overline Z(s) )+O(\e^{\MH 1-\delta}) \big)ds}\leq \f{1}{\rho_\e(0)}e^{-\f{a_0t}{\e^2}}.
$$
To control the second term in the r.h.s. of \fer{eq:Je} we let $t\geq \e^{\beta}$, with $\beta\in {(1,2)}$, and we write
$$
\begin{aligned}
\f{\kappa}{\e^2}\int_0^t e^{-\f{1}{\e^2}\int_\tau^t \big( r-m(\overline Z(s))  +O(\e^{\MH 1-\delta}) \big) ds}d\tau
=&
\f{\kappa}{\e^2}\int_0^{t-\e^\beta} e^{-\f{1}{\e^2}\int_\tau^t \big( r-m(\overline Z(s))  +O(\e^{\MH 1-\delta}) \big) ds}d\tau\\
&+\f{\kappa}{\e^2}\int_{t-\e^\beta}^t e^{-\f{1}{\e^2}\int_\tau^t \big( r-m(\overline Z(s))  +O(\e^{\MH 1-\delta}) \big) ds}d\tau\\
=&J_{\e,1}+J_{\e,2}.
\end{aligned}
$$
We control each of these terms separately. To control the first term we write
\begin{align*}
J_{\e,1}& =\f{\kappa}{\e^2}\int_0^{t-\e^\beta} e^{-\f{1}{\e^2}\int_\tau^t \big( r-m(\overline Z(s))  +O(\e^{\MH 1-\delta}) \big) ds}d\tau\leq
\f{\kappa}{\e^2}\int_0^{t-\e^\beta} e^{-\f{1}{\e^2}a_0(t-\tau)}d\tau\\
& \leq \f{\kappa}{a_0} \big( e^{-\f{a_0}{\e^{2-\beta}}}-e^{-\f{a_0t}{\e^2}} \big) \leq \f{\kappa}{a_0}  e^{-\f{a_0}{\e^{2-\beta}}}.
\end{align*}
To control $J_{\e,2}$ we use the fact that
$$
{
|m(\overline Z(s))-m(\overline Z(t))|\leq C\e^{\beta},\qquad \text{for all $s\in [t-\e^{\beta},t]$}.
}
$$
 Up to changing the constant $C$, we deduce that, {since $\beta > 1:$}
$$
\begin{aligned}
J_{\e,2}(t)=&\f{\kappa}{\e^2}\int_{t-\e^\beta}^t e^{-\f{1}{\e^2}\int_\tau^t \big( r-m(\overline Z(s))  +O(\e^{\MH 1-\delta}) \big) ds}d\tau\\
\leq&
\f{\kappa}{\e^2}\int_{t-\e^\beta}^t e^{-\f{1}{\e^2}  \big( r-m(\overline Z(t))  -C\e^{\MH 1-\delta} \big) (t-\tau)}  d\tau \\
\leq& \f{\kappa}{r-m(\overline Z(t))-C\e^{\MH 1-\delta}} .
\end{aligned}
$$
We  similarly find
$$
\begin{aligned}
J_{\e,2}(t)=&\f{\kappa}{\e^2}\int_{t-\e^\beta}^t e^{-\f{1}{\e^2}\int_\tau^t \big( r-m(\overline Z(s))  +O(\e^{\MH 1-\delta}) \big) ds}d\tau\\
\geq&
\f{\kappa}{\e^2}\int_{t-\e^\beta}^t e^{-\f{1}{\e^2} \big( r-m(\overline Z(t))  +C\e^{\MH 1-\delta} \big)  (t-\tau)}  d\tau \\
=& \f{\kappa}{r-m(\overline Z(t))+C\e^{\MH 1-\delta}}\big( 1-e^{-\f{ r-m(\overline Z(t))  +C\e^{\MH 1-\delta}  }{\e^{2-\beta}}}\big).
\end{aligned}
$$
Using again \fer{as:M10},  \fer{eq:Ze} and  {(H2)} , when $\e$ is chosen small enough, we have
$$
 r-m(\overline Z(t))  +C\e^{\MH 1-\delta}\geq a_0/2, 
 $$
 which implies that
 $$
 J_{\e,2}(t)\geq \f{\kappa}{r-m(\overline Z(t))+C\e^{\MH 1-\delta}}\big( 1-e^{-\f{a_0}{2\e^{2-\beta}}} \big),
 $$
and hence
$$
|J_{\e,2}(t))-\f{\kappa}{r-m(\overline Z(t))}|=O(\e^{\MH 1-\delta}).
$$
Combining the inequalities above we obtain that
$$
|J_\e(t)-\f{\kappa}{r-m(\overline Z(t))}|\leq O(\e^{\MH 1-\delta})+\big( \f{1}{\rho_\e(0)} + \f{\kappa}{a_0} \big) e^{-\f{a_0}{\e^{2-\beta}}} ,\qquad \text{for all $t\geq \e^\beta$} .
 $$
 We deduce that, {for all $t\geq \e^\beta$},
 $$
 |\rho_\e(t)-\f {r-m(\overline Z(t))}{\kappa}|\leq {\rho_\e(t)\left(\f {r-m(\overline Z(t))}{\kappa}\right)} \Big( O(\e^{\MH 1-\delta})+ \big( \f{1}{\rho_\e(0)} + \f{\kappa}{a_0} \big) e^{-\f{a_0}{\e^{2-\beta}}}\Big).
 $$
 We next apply the maximum principle to \fer{eq-rho} and use Assumption \fer{as:rho0} to obtain that
 $$
 \rho_\e(t)\leq \max(\rho_M,r), \qquad \text{for all $t\geq 0$}.
 $$
 We also recall from Assumption \fer{as:rho0} that $\rho_\e(0)\geq \rho_m>0$. We conclude that
$$
 \big| \rho_\e(t)-\f {r-m(\overline Z(t))}{\kappa} \big|=O(\e^{\MH 1-\delta}) , \qquad {\text{for all $t\geq \e^\beta$}.}
 $$
 
 \medskip
 
{ (b) {\bf The convergence of $n_\e$ as $\e\to 0$.} From the definition of $q_\e$ and thanks to the convergence of $\rho_\e$ which was established in the paragraph above, it is enough to prove that $q_\e$ converges in $C((0,\infty);\mathcal M^1(\R))$ to $\partial(x-\overline Z(t))$.  We next notice from the definition of $g_\e$ that it converges in $C((0,\infty);\mathcal M^1(\R))$ to $\partial(x-\overline Z(t))$. The idea is then to use the estimate \fer{est-main-Was} to prove that $q_\e$ and $g_\e$ have the same limit.
}

{Let $\phi:\R\to \R$  be a  Lipschitz continuous  function. We compute, for all $t\in \R^+$ and using \fer{est-main-Was},
\beq
\label{est-phiqg}
\big|\int_\R \phi(x)(q_\e(t,x)-g_\e(t,x))dx \big|\leq \|\phi'\|_{L^\infty}W_1(q_\e,g_\e)\leq  K \|\phi'\|_{L^\infty} \left( W_1(q_{\e,0},g_{\varepsilon,0}) + \varepsilon^{1-\delta} \right).
\eeq
We next estimate $W_1(q_{\e,0},g_{\varepsilon,0})$. To this end, we will use the following triangular inequality
$$
W_1(q_{\e,0}, g_{\varepsilon,0} ) \leq W_1(q_{\e,0}, \partial(\cdot-x_0))+W_1(g_{\e,0}, \partial(\cdot-x_0)).
$$
 Again from the definition of $g_\e$, it is immediate that the second term in the r.h.s. of the inequality above tends to $0$, as $\e\to 0$. We prove the same property for the first term. We first compute, using the Cauchy-Schwartz inequality,
 $$
  W_1(q_{\e,0}, \partial(\cdot-M_{\e,1}(0) ) ) \leq \int |y-M_{\e,1}(0)| q_{\e,0}(y)dy \leq \Big(\int |y-M_{\e,1}(0)|^2 q_{\e,0}(y)dy \Big)^{1/2}=\sqrt{M_{\e,2}(0)}.
 $$
 We then use Assumption \fer{eq_initialdata} to obtain that 
 $$
  W_1(q_{\e,0}, \partial(\cdot-M_{\e,1}(0) ) ) \to 0, \qquad \text{as $\e\to 0$}.
  $$
 Combining this property with Assumption \fer{as:rho0} we deduce that 
  $$
  W_1(q_{\e,0}, \partial(\cdot-x_0) ) \to 0, \qquad \text{as $\e\to 0$}.
  $$
  and hence 
  $$
  W_1(q_{\e,0}, g_{\varepsilon,0} ) \to 0, \qquad \text{as $\e\to 0$}.
  $$
 This concludes the proof thanks to \fer{est-phiqg}.
}

 \section{A comparison with a model with an asexual reproduction term}
 \label{sec:comp-asex}
 
 A closely related model, considering an asexual reproduction and including mutations, is written
 \begin{equation} \label{eq_asex}
\begin{cases}
\varepsilon  \partial_t   n_\e = p M_{\varepsilon}  [ n_\e]  
+  \left({ r-m(x) }- \kappa \rho_\e(t) \right) n_\e ,\\
\rho_\e(t)= \int_{\mathbb R}   n_\e(t,y)dy,\\
n_\e(0,x)=n_{\e,0}(x).
\end{cases}
\end{equation}
where
\[
M_{\varepsilon}[n_\e](x) =\int G_\e(x-y) \big( n_\e(t,y)-n_\e(t,x) \big)  dy ,
\qquad
G_\e(x )= \f{1}{\e } G(\f{x}{\e}).
\]
A typical example of $G_\e$ is a normal distribution with variance of order $\e^2$, for instance $\Gamma_\e$.
Here, since each offspring has a single parent, contrary to model \fer{eq_main}  the reproduction is modeled via linear terms.
{The parameter $r$ corresponds to the clonal birth rate of individuals of trait $x$ and   $m(x)$ represents a trait-dependent death rate.} The term $M_\e$ models the mutations which arise with rate $p$. The function $G_\e$ corresponds to the mutation law.  The small parameter $\e$ is introduced to rescale the mutation size, such that we consider small mutational effects. It is shown in previous works \cite{GB.BP:08,GB.SM.BP:09,AL.SM.BP:10} and under suitable assumptions that, as $\e\to0$, the phenotypic density $n_\e$ has a similar behavior to the one observed in Theorem \ref{thm:main} in the case of model \fer{eq_main}.
It indeed   converges to a measure $n$, with
$$
n(t)=\rho(t)\partial(x-\overline x(t)).
$$
Under suitable assumptions \cite{AL.SM.BP:10,SM.JR:16}, the dynamics of $\overline x(t)$ can also be described via an ordinary differential equation
$$
\f{d}{dt}\overline x(t)= {-A(t)  m'(\overline x(t))}.
$$
 This equation can be compared to \fer{eq:Z}, with the difference that here the coefficient  in front of the gradient of the growth rate,   $A(t)$, is time dependent. This dependency with respect to time comes from the fact that the dominant term of the variance of the phenotypic distribution corresponding to \fer{eq_asex} varies with time, unlike to the model \fer{eq_main} (see \fer{eq_M2close}). The dynamics of this coefficient $A(t)$ is rather complex and can be obtained via a Hamilton-Jacobi equation with constraint. The usual method to study \fer{eq_asex} is indeed via an approach based on Hamilton-Jacobi equations, the first step being to perform a Hopf-Cole transformation 
 $$
 n_\e(t,x)=\exp \big( \f{u_\e(t,x)}{\e} \big).
 $$
 It can be shown \cite{GB.BP:08,GB.SM.BP:09} that $u_\e$ converges, as $\e\to 0$, to a continuous function $u$, which solves the following Hamilton-Jacobi equation in the viscosity sense
\beq
\label{HJ}
 \begin{cases}
{ \p_t u (t,x)= \displaystyle {p\int_\R G(h)e^{\f{\p}{\p x} u(t,x)h}dh + r-m(x)}-\kappa \rho(t),}& (t,x)\in \R^+\times \R,\\
 \max_{x\in \R} u(t,x)=0,& x\in \R,\\
 u(0,x)=u_0(x),
 \end{cases}
\eeq
 where $\rho$ is the limit of $\rho_\e$ as $\e\to 0$.
 The function $u$ relates to the phenotypic density $n$ via the following relation
 $$
 \mathrm{supp}\ n(t,\cdot) \subset \{ x\, |\, u(t,x)=0 \}.
 $$
 Under some global concavity assumptions \cite{AL.SM.BP:10,SM.JR:16} one can prove that the set above has a single point $\overline x(t)$ and hence 
 $$
 n(t,x)=\rho(t) \partial(x-\overline x(t)).
 $$
 The coefficient $A(t)$ is then given by the following formula 
 $$
 A(t)=-\f{1}{\f{\p^2}{\p x^2}u(t,\overline x(t))}.
 $$
One can wonder whether the moment-based approach introduced in this article can allow to obtain similar results in the case of model \fer{eq_asex} providing in this way a direct proof, rather than using Hamilton-Jacobi equations. We will explain below why the complexity of the dynamics in the model with asexual reproduction does not allow to apply the direct method presented above.

 In order to keep the computations more straightforward and to highlight the main arguments we will focus on the following particular death rate
$$
{m(x)=sx^2.}
$$
We also assume that $G$ is symmetric, that is
$$
G(x)=G(-x),\qquad \text{for all $x\in \R$}.
$$
{We then define $q_\e$ similarly to Section \ref{sec:intro}:
$$
q_\e(t,x)=\f{n_\e(t,x)}{\rho_\e(t)}.
$$
Replacing this in \fer{eq_asex} we obtain
 \begin{equation} \label{eq_asex_qe}
\begin{cases}
\varepsilon  \partial_t   q_\e =  \displaystyle p M_{\varepsilon}  [ q_\e]  
+  \left({ -m(x) }+ \int_{\mathbb R} m(y)q_\e({\MH \cdot},y)dy\right) q_\e ,\\
q_\e(0,x)=\f{n_{\e,0}(x)}{\rho_\e(0)}.
\end{cases}
\end{equation}
We also define $M_{\e,i}$ and $M_{\e,i}^c$ similarly to Section \ref{sec:intro}.} We next derive the equations on the first and the second order moments $M_{\e,1}$ and $M_{\e,2}^c$, similarly to Section \ref{sec:moments}.

To derive an equation on $M_{\e,1}$, we first compute
$$
\begin{aligned}
\int_\R xM_\e[q_\e](x)dx=&\int_\R\int_\R x\big( q_\e({\MH\cdot,}x+\e h)-q_\e({\MH\cdot,}x)\big) G(h)dhdx\\
=& \int_\R\int_\R (x+\e h)  q_\e({\MH\cdot,}x+\e h)G(h)dhdx- \int_\R\int_\R x  q_\e({\MH\cdot,}x)G(h)dhdx\\
&-\e \int_\R \int_\R hq_\e({\MH\cdot,}x+\e h)G(h)dhdx\\
=&-\e \int_\R h G(h)dh=0.
\end{aligned}
$$
Note that here we have used the fact that $G$ is symmetric. We next multiply   {\fer{eq_asex_qe}} by $x$ and integrate with respect to $x$ to obtain
$$
\begin{aligned}
\e \dot{M}_{\e,1}=&{-s\int_\R  x^3 q_\e({\MH\cdot,}x)dx+s\int_\R  x^2 q_\e({\MH\cdot,}x)dx\, {M}_{\e,1}}\\
=&-s\int_\R  (x-M_{\e,1}+M_{\e,1})^3 q_\e({\MH\cdot,}x)dx+s\int_\R  (x-M_{\e,1}+M_{\e,1})^2 q_\e({\MH\cdot,}x)dx\, {M}_{\e,1}\\
=&-2s M_{\e,1} M_{\e,2}^c-sM_{\e,3}^c
\\
=&-m'(M_{\e,1}) M_{\e,2}^c-sM_{\e,3}^c.
\end{aligned}
$$
Therefore, $M_{\e,1}$ solves an equation with similar structure to \fer{eq_M1}, that is 
$$
\e \dot{M}_{\e,1}{+m'(M_{\e,1})} M_{\e,2}^c=F_1={-}sM_{\e,3}^c,
$$
where the r.h.s. is a priori of upper order term  in $\e$. The equation on $M_{\e,1}$ is not indeed the one that leads to  difficulty.

We next compute the equation solved by $M_{\e,2}^c$. To this end, we first compute
$$
\begin{aligned}
\int_\R (x-M_{\e,1})^2 M_\e[q_\e](x)dx =& \int_\R \int_\R (x-M_{\e,1})^2\big( q_\e({\MH\cdot,}x+\e h)-q_\e({\MH\cdot,}x)\big) G(h)dhdx\\
=&\int_\R \int_\R x^2\big( q_\e({\MH\cdot,}x+\e h)-q_\e({\MH\cdot,}x)\big) G(h)dhdx\\
=&\int_\R\int_\R (x+\e h-\e h)^2 q_\e({\MH\cdot,}x+\e h)G(h)dh {dx}- \int_\R\int_\R x^2 q_\e({\MH\cdot,}x)G(h)dh{dx}\\
=&-2 \e M_{\e,1}\int_{\mathbb R} hG(h)dh+\e^2 \int_\R h^2 G(h)dh\\
=&\e^2 \sigma,
\end{aligned}
$$
where $\sigma$ denotes the variance of the distribution $G$. We next multiply {\fer{eq_asex_qe}} by $(x-M_{\e,1})^2$ and integrate with respect to $x$ to obtain
$$
\begin{aligned}
\e \dot{M}_{\e,2}^c
=&p\sigma \e^2-s\int_\R (x-M_{\e,1})^2 x^2  q_\e({\MH\cdot,}x)dx+s\int_\R  x^2 q_\e({\MH\cdot,}x)dx\, {M}_{\e,2}^c\\
=&p\sigma \e^2-s\int_\R (x-M_{\e,1})^2 (x-M_{\e,1}+M_{\e,1})^2  q_\e({\MH\cdot,}x)dx \\
& +s\int_\R  (x-M_{\e,1}+M_{\e,1})^2 q_\e({\MH\cdot,}x)dx\, {M}_{\e,2}^c.
\end{aligned}
$$
We deduce that 
\beq
\label{eq:M2-asex}
\e \dot{M}_{\e,2}^c=p\sigma \e^2-sM_{\e,4}^c-2sM_{\e,3}^cM_{\e,1}+s(M_{\e,2}^c)^2.
\eeq
Here the structure of the equation is different from \fer{eq_M2}. The dissipative term that allowed us to close the equation is not present here ({\em i.e.} the equation is not written in the form $\e\dot{M}_{\e,2}^c+ CM_{\e,2}^c=F$ {with $C\in (0,\infty)$} ). Moreover, all the terms on the r.h.s. of \fer{eq:M2-asex} may potentially be of the same order and could equally contribute to the dynamics of $M_{\e,2}^c$ (see {\cite{SM.JR:15-1,SM:17,SM.JR:20}}). We hence do not expect to be able to close the equations of the moments at this stage. The dynamics of moments of order $3$ and $4$ are indeed driven by the solution to the Hamilton-Jacobi equation \fer{HJ} {\cite{SM:17,SM.JR:20}}.

\appendix
\section{The Cauchy problem; the proof of Proposition \ref{prop:Cauchy}}
\label{sec:appendix}
In this section, we focus on the Cauchy theory analyzed on { Proposition \ref{prop:Cauchy}}. This issue is tackled previously in \cite{GR:17}, when the {trait dependency is only} included in the reproduction kernel $T_{\varepsilon}$, or in \cite{Prevost} on bounded domains. For completeness, we explain here how the arguments may be adapted to our framework. We shall in particular detail the issues related to the boundedness of $m.$ This rate is classically assumed quadratic and thus unbounded while this makes the analysis surprisingly more difficult.  

In the whole section we fix  $\varepsilon=1$ since this has no influence on our computations and we drop all $\varepsilon$ in notations. Also, we focus on the case of a smooth mortality rate $m$ satisfying (H0) and (H3) (the other assumptions have no influence here).
 The outline of the appendix is as follows. Firstly, we provide a definition of solution to \eqref{eq_main} and analyze its properties. 
Then, we show that existence of solutions in case $m$ is bounded reduces to a classical application of Cauchy-Lipschitz theorem and tackle the existence of solutions in case $m$ is unbounded. We conclude by obtaining uniqueness of solutions with a sufficiently large number of bounded moments.

\subsection{What is a solution to \eqref{eq_main} ?}
From now on $k \in (1,\infty)$ is fixed. We denote
then $w_{k}(x) = (1+|x|)^k$ and 
\[
X_{k} := \left\{ n \in L^1(\mathbb R) \text{ s.t. } \int_{\mathbb R} w_k(x) |n(x)|{\rm d}x < \infty \right\}.
\]
This is a Banach space endowed with the norm:
\[
\|n\|_{X_k} := \int_{\mathbb R} w_k(x) |n(x)|{\rm d}x.
\]
Inside $X_k$ we denote:
\[
\mathcal U_k := \left\{ n \in X_k \text{ s.t. } \int_{\mathbb R} n(x){\rm d}x \neq 0 \right\}.
\]
Clearly $\mathcal U_k$ is an open subset of $X_k.$ 
We denote also below $C_b(\mathcal O)$ the set of bounded continuous functions on $\mathcal O$ endowed with its classical $L^{\infty}$-norm.
We recall the definition of the mapping $T$: 
\[T[n(t,\cdot)](x) =\int_{\mathbb R} \int_{\mathbb R} \Gamma \left(x-\frac{(y+y')}{2}\right)  n(t,y)\f{n(t,y')}{\rho(t)} dydy'.\]
For the following analysis, we  define solutions to 
\eqref{eq_main} as follows
\begin{definition} \label{def_solution}
Given $n_0 \in X_k \cap C_b(\mathbb R)$ and $T_0 >0,$   we say that $n$ is a solution to \eqref{eq_main}  on $(0,T_0)$ if
\begin{itemize}
\item $n \in C_b([0,T_0] \times \mathbb R) \cap L^{\infty}((0,T_0);X_k)$
with $n(t,\cdot) \in \mathcal U_k$ for all $t \in (0,T_0),$ 
\item for all $x \in \mathbb R$ the restriction $n(\cdot,x)$ solves \eqref{eq_main} in $\mathcal D'((0,T_0)).$
\end{itemize}
We say that $n$ is a global solution if its restriction on $[0,T_0]$ is a solution on $(0,T_0)$ for arbitrary $T_0>0.$
\end{definition}

We point out that we do not assume {\em a priori} that $n_0 \geq 0$ here.  However, we shall restrict in the second part
of the analysis to such initial data.  A straightforward interpolation argument entails also that  a solution $n$ satisfies $n \in C([0,T_0];X_{\tilde{k}})$
for arbitrary $\tilde{k} < k.$ We also point out that the introduction of distributions $\mathcal D'((0,T_0))$ is just for convenience here. Indeed, under the (satisfied) condition that $n \in C([0,T_0] ; L^1(\mathbb R))$ has non-zero integral, we observe that $t \mapsto T[n(t,\cdot)](x) \in C([0,T_0])$ so that equation  $n(\cdot,x) \in C^1([0,T_0])$ and \eqref{eq_main} holds in a classical sense.
 
 \medskip
 
 To analyze the Cauchy problem, 
we will rely on the following remark
\begin{lemma} \label{lem_Tlip}
The mapping $T :  \mathcal U_k \cap C_b(\mathbb R) \to X_k \cap C_b(\mathbb R)$ is positive and locally lipschitz. More precisely, 
\begin{itemize}
\item[i)] given $(n,\tilde{n}) \in (\mathcal U_k \cap C_b(\mathbb R))^2,$  there exists a constant  
$K$ depending increasingly on :
\[
\max \left(\f{\| n \|_{X_k}}{\left|\int_{\mathbb{R}} n(z){\rm d}z\right|},\f{\| \tilde{n} \|_{X_k}}{\left|\int_{\mathbb{R}}\tilde{n}(z){\rm d}z\right|}\right)
\] 
such that:
\[
\| T[n] - T[\tilde{n}] \|_{X_k \cap C_b(\mathbb R)} \leq K \| n- \tilde{n} \|_{X_k},
\]
\item[ii)] if $n \geq 0$ we have $T[n] \geq 0.$
\end{itemize}
\end{lemma}
\begin{proof}
Property ii) follows from the fact that all the quantities in the integrals are nonnegative.  For property i), we first get:
\[\begin{aligned}
\| T[n] - T[\tilde{n}] \|_{X_k} & \leq \int_{\mathbb R}\int_{\mathbb R}\int_{\mathbb R} \Gamma \left(x-\frac{(y+y')}{2}\right)w_k(x)\left| n(y)\f{n(y')}{\int_{\mathbb R} n(z){\rm d}z} - \tilde{n}(y)\f{\tilde{n}(y')}{\int_{\mathbb R} \tilde{n}(z){\rm d}z} \right| {\rm d}y {\rm d}y'{\rm d}x.
\end{aligned} \] 
{We next notice that} 
\[w_k(x) \leq C_k \left(w_k \left(x-\frac{(y+y')}{2}\right) +w_k \left(y\right)+w_k(y')\right), \]
with $C_k$ a constant depending on $k$.
Moreover, {we have}
\[\begin{aligned}
n(y)\f{n(y')}{\int_{\mathbb R} n(z){\rm d}z} - \tilde{n}(y)\f{\tilde{n}(y')}{\int_{\mathbb R} \tilde{n}(z){\rm d}z} = &\f{n(y)}{\int_{\mathbb R} n(z){\rm d}z}(n(y')-\tilde{n}(y'))+\f{\tilde{n}(y')}{\int_{\mathbb R} \tilde{n}(z){\rm d}z}(n(y)-\tilde{n}(y))\\ 
&+\f{n(y)}{\int_{\mathbb R} n(z){\rm d}z}\f{\tilde{n}(y')}{\int_{\mathbb R} \tilde{n}(z){\rm d}z}\int_{\mathbb{R}}(n(z)-\tilde{n}(z) ){\rm d}z  
\end{aligned}\]
{We deduce},  since $w_k \Gamma \in L^1(\mathbb R),$ that there is a constant $C_{k,\Gamma}$ depending further on $\Gamma$ for which: 
\[\begin{aligned}
\| T[n] - T[\tilde{n}] \|_{X_k}  \leq & C_{k,\Gamma} \int_{\mathbb R}\int_{\mathbb R} (w_k \left(y\right)+w_k(y'))\left| \f{n(y)}{\int_{\mathbb R} n(z){\rm d}z}(n(y')-\tilde{n}(y')) \right|  {\rm d}y {\rm d}y' \\
&+ C_{k,\Gamma} \int_{\mathbb R}\int_{\mathbb R} (w_k \left(y\right)+w_k(y'))\left| \f{\tilde{n}(y')}{\int_{\mathbb R} \tilde{n}(z){\rm d}z}(n(y)-\tilde{n}(y)) \right|  {\rm d}y {\rm d}y' \\
&+ C_{k,\Gamma}\int_{\mathbb R}\int_{\mathbb R} (w_k \left(y\right)+w_k(y'))\left| \f{n(y)}{\int_{\mathbb R} n(z){\rm d}z}\f{\tilde{n}(y')}{\int_{\mathbb R} \tilde{n}(z){\rm d}z}\int_{\mathbb{R}}(n(z)-\tilde{n}(z) ){\rm d}z   \right|  {\rm d}y {\rm d}y'\\
\leq & C_{k,\Gamma}\left(\f{\| n \|_{X_k}}{\left|\int_{\mathbb{R}} n(z){\rm d}z\right|} + \f{\| \tilde{n} \|_{X_k}}{\left|\int_{\mathbb{R}}\tilde{n}(z){\rm d}z\right|}+\f{\| n \|_{X_k}}{\left|\int_{\mathbb{R}} n(z){\rm d}z\right|}\f{\| \tilde{n} \|_{X_k}}{\left|\int_{\mathbb{R}}\tilde{n}(z){\rm d}z\right|}\right) \| n- \tilde{n} \|_{X_k}\\
\leq & C_{k,\Gamma}\f{\| n \|_{X_k}}{\left|\int_{\mathbb{R}} n(z){\rm d}z\right|}\f{\| \tilde{n} \|_{X_k}}{\left|\int_{\mathbb{R}}\tilde{n}(z){\rm d}z\right|}\| n- \tilde{n} \|_{X_k}.
\end{aligned}
\]
We obtain with similar and even simpler computations that there is a constant $C_{k,\Gamma}$ for which: 
\[
\| T[n] - T[\tilde{n}] \|_{C_b(\mathbb R)} \leq  C_{k,\Gamma}\f{\| n \|_{X_k}}{\left|\int_{\mathbb{R}} n(z){\rm d}z\right|}\f{\| \tilde{n} \|_{X_k}}{\left|\int_{\mathbb{R}}\tilde{n}(z){\rm d}z\right|}\| n- \tilde{n} \|_{X_k}.
\]{\qed}
\end{proof}

This lemma entails that the mapping $T$ extends to a continuous mapping $C([0,T_0] ; X_k)  \cap C_b([0,T_0] \times \mathbb R)) \to C([0,T_0]; X_k) \cap C_b([0,T_0] \times \mathbb R).$ We point out also that a straight forward adaptation of the above proof entails the following pointwise bound {\MH for nonnegative $n \in \mathcal U_k$:}
\begin{equation} \label{eq_Cbw}
|T[n](x)| \leq \dfrac{C_{k}}{w_k(x)} \|n\|_{X_k}  \qquad 
\end{equation}
where $C_{k}$ is a constant independent of $n .$  In particular,  if $n$ is {\MH a nonnegative} solution to \eqref{eq_main} 
on $(0,T_0)$ with initial data $n_0 \in \mathcal U_k$ we have that
$\partial_t n(t,x) \in C([0,T_0])$ for all $x \in \mathbb R$ and combining the latter inequality with \eqref{eq_bis} we infer that:
\begin{equation} \label{eq_controldtn}
|\partial_t n(t,x)| \leq \dfrac{C_{k}}{w_k(x)} \|n(t,\cdot)\|_{X_k} + (A_m w_{p+2}(x) + \kappa \rho) n(t,x). 
\end{equation}
Given $j < k - (p+3)$,  classical parameter integral arguments entail in particular that
\[
M_j(t) := \int_{\mathbb R} x^j n(t,x){\rm d}x  \in C^1([0,T_0])
\]
with:
\[
\dot{M}_j(t) = \int_{\mathbb R} x^j \partial_t n(t,x){\rm d}x  \quad \forall \, t \in [0,T_0].
\]
We use without mention this regularity property in the main body of  the article.

\subsection{Existence of solutions}
In case $m$ is bounded,  the mapping 
\[
n \mapsto \left(m + \kappa \int_{\mathbb R} n \right)n 
\]
is clearly locally lipschitz on $\mathcal U_k.$ Combining this remark with {Lemma \ref{lem_Tlip}}, we obtain via a direct application of the Cauchy-Lipschitz theorem that,  given $n_0 \in \mathcal U_k.$ 
there exists $T_0$ depending decreasingly on 
\[
{\MH \|n_0\|_{X_k} \left( 1 + \dfrac{1}{\int_{\mathbb R} n_0} \right)}
\]
such that for arbitrary $\tilde{T}_0 < T_0$ there is a unique solution $n$ to \eqref{eq_main} on $(0,\tilde{T}_0).$ 
We point out that the Cauchy-Lipschitz theorem yields also the further regularity $n \in C^1([0,\tilde{T}_0);X_k).$ 
Furthermore, it entails the existence of a unique non-extendable solution defined on $[0,\tilde{T}_0)$ with the blow-up alternative:
\begin{itemize}
\item either $\tilde{T}_0 = +\infty$
\item or $\tilde{T}_0 < \infty$ and 
\[
\limsup_{t \to \tilde{T}_0} \dfrac{1}{\int_{\mathbb R} n(t,x){\rm d}x} + \|n(t,\cdot)\|_{X_k} = +\infty
\]
\end{itemize}

\subsubsection{Quantitative estimates and global existence result}
To yield global existence of solution we now restrict to non-negative initial data. We prove {\em a priori} estimates satisfied by solutions to \eqref{eq_main}.  We point out that we derive estimates satisfied by any solution to \eqref{eq_main} according to {Definition \ref{def_solution}} not necessarily the ones obtained through the Cauchy-Lipschitz theorem above.  Moreover, we derive estimates that are valid for mortality rates $m$ that possibly diverge at infinity.
In the remainder of this subsection, we assume that $n_0 \in \mathcal U_{k},$ for some $k >0,$ is nonnegative. We note that $n_0$ has non-zero {integral} so that it may not vanish identically on $\mathbb R.$ We assume that $n$ is  a solution to \eqref{eq_main} on $(0,T)$ according to {Definition \ref{def_solution}}.  

\medskip

Firstly, we show that positivity propagates in the solution:
\begin{lemma}
{We assume that $n_0(x)\geq 0$ for all $x\in \mathbb R$.} There holds:
\[
n(t,x) \geq  0, \quad \forall \, (t,x) \in [0,T] \times \mathbb R. 
\]

\end{lemma}
\begin{proof}
We show that $n$ remains non-negative on a small time interval $[0,T_0].$ The result extends to $[0,T]$ with a standard connectedness argument.  {\MH Indeed,  our argument propagates positivity as long as $n(t,\cdot)$ does not vanish identically. However,  if $n(t,\cdot) \equiv 0$ then the solution blows up in $t$.}
We split the proof into three steps.

\paragraph{\em Step 1.}  We obtain first a general property for the operator $T$ on $L^1(\mathbb R).$ Namely, given $f \in L^1(\mathbb R) \cap C_b(\mathbb R)$ with  positive {integral}, we have clearly via standard parameter-integral results that $T[f] \in C_b(\mathbb R)$.  We prove that:
\begin{equation} \label{eq_negT}
\| (T[f])_- \|_{L^{\infty}(\mathbb R)} \leq   \dfrac{4\|f\|_{L^1(\mathbb R)}}{\int_{\mathbb R} f(x){\rm d}x} \| (f)_{-}\|_{L^{\infty}(\mathbb R)},  
\end{equation}
where $(\cdot)_-$ denotes the negative part of continuous functions. Indeed, if we denote:
\[
D_- := \{x \in \mathbb R \text{ s.t. } f(x) < 0\} \qquad D_+ := \{x \in \mathbb R \text{ s.t. } f(x) \geq 0\}
\]
we remark that $f(y) f(y') \geq 0$ whenever $(y,y') \in D_- \times D_- \cup D_+ \times D_+.$ Therefore, we have:
\[
\begin{aligned}
(T[f])_-(x) &\leq - \int_{D_- \times D_+} \Gamma \left(x- \frac{y+y'}{2} \right) \dfrac{f(y) f(y')}{\int_{\mathbb R} f(z){\rm d}z}{\rm d}y {\rm d}y' +sym.  \\
                & \leq \dfrac{1}{\int_{\mathbb R} f(z){\rm d} z} \left(  \int_{D_- \times D_+} \Gamma \left(x- \frac{y+y'}{2} \right) {|f(y)| |f(y')|} {\rm d}y {\rm d}y + sym. \right) \\
                & \leq \dfrac{1}{\int_{\mathbb R} f(z){\rm d} z} \left(  \int_{D_- \times \mathbb R} \Gamma \left(x- \frac{y+y'}{2} \right) {\|(f)_-\|_{L^{\infty}(\mathbb R)} |f(y')|} {\rm d}y {\rm d}y + sym. \right) 
\end{aligned} 
\]
We conclude then by recalling that:
\[
\int_{\mathbb R} \Gamma \left(x- \frac{y+y'}{2} \right) {\rm d}y = 2.
\]

\paragraph{\em Step 2.} We consider now $n_0$ and its associated solution $n.$ Thanks to the remark after {\bf Definition \ref{def_solution}}, we note that
$n \in C([0,T]; L^1(\mathbb R))$ with $n(0,\cdot) = n_0.$ In particular, we can construct a small time $T_0$ such that:
\begin{equation} \label{eq_boundn}
\|n(t,\cdot)\|_{L^1(\mathbb R)} \leq 2 \|n_0\|_{L^1(\mathbb R)} \qquad 0 < \dfrac{\int_{\mathbb R} n_0(z){\rm d}z}{2} \leq \int_{\mathbb R}n(t,z){\rm d}z \qquad
\forall \, t \in [0,T_0]. 
\end{equation}
We choose $\Lambda_0 >32 r.$ Since $n_0 \geq 0$ we note that we have in particular:
\[
\Lambda_0 > 32r \dfrac{\|n_0\|_{L^1(\mathbb R)}}{\int_{\mathbb R} n_0(z){\rm d}z}.
\]
We construct then:
\[
\tilde{n}(t,x) = n(t,x) \exp\left( -\Lambda_0 t\right) \qquad \forall \, t \in [0,T].
\]
Like $n,$ we have that $\tilde{n} \in C_b([0,T]\times \mathbb R) \cap C([0,T] ; L^1(\mathbb R))$ and that
$\tilde{n}(\cdot, x) \in C^1([0,T])$ for all $x \in \mathbb R.$ In particular, for arbitrary $x \in \mathbb R,$ there holds:
\[
\partial_t \tilde{n} (t,x) = (\partial_tn (t,x) - \Lambda_0 n(t,x)) \exp(-\Lambda_0 t).
\]
Replacing $\partial_t n$ and arguing that $T[\cdot]$ is $1$-homogeneous, we obtain that:
\begin{equation} \label{eq_dtntilde}
\partial_t \tilde{n}(t,x) = r T[\tilde{n}(t,\cdot)](x) - (m(x) + \kappa \rho(t) + \Lambda_0) \tilde{n}(t,x) \text{ on $[0,T].$} 
\end{equation}

\paragraph{\em Step 3.} We assume now that $n$ becomes strictly negative on $[0,T_0]$ or equivalently that $\tilde{n}$ takes strictly negative values on $[0,T_0]. $ Since $\tilde{n} \in C_b([0,T_0] \times \mathbb R),$ we have:
\[
\max _{[0,T_0] \times \mathbb R}( \tilde{n})_-(t,x) = \delta > 0,
\]
and we can find $(t_0,x_0) \in (0,T_0) \times \mathbb R$ so that $\tilde{n}(t_0,x_0) \leq - \delta/2.$ Since $\tilde{n}(\cdot,x_0) \in C^1([0,T_0])$ satifies $\tilde{n}(0,x_0) = n_0(x_0) \geq 0$ we can then construct a time $t_- \in (0,T_0)$ such that $\tilde{n}(t,x_0) \geq \tilde{n}(t_-,x_0) = -\delta/2$ for all $t \leq t_-.$
In praticular, there holds  :
\[
\partial_t \tilde{n}(t_-,x_0) \leq 0.
\]
Replacing with \eqref{eq_dtntilde} and recalling that $\tilde{n}(t_-,x_0) = -\delta/2$ we infer that:
\[
r (T[\tilde{n}])_- (t_-,x_0) \geq ( m + \kappa \rho + \Lambda_0 )  \dfrac{\delta}{2} \geq \dfrac{\Lambda_0}{2} \delta > 16 r \delta.
\]
However, we can apply \eqref{eq_negT} to $f = \tilde{n}(t_-,\cdot)$ and we obtain:
\[
\begin{aligned}
r (T[\tilde{n}])_-(t_-,x_0) &  \leq \dfrac{4r \|\tilde{n}(t_-,\cdot)\|_{L^1(\mathbb R)}}{\int_{\mathbb R} \tilde{n}(t_-,z){\rm d}z} \|(\tilde{n})_-(t_-,\cdot)\|_{L^{\infty}(\mathbb R)}\\
&\leq \dfrac{4r \|{n}(t_-,\cdot)\|_{L^1(\mathbb R)}}{\int_{\mathbb R} {n}(t_-,z){\rm d}z}\|(\tilde{n})_-\|_{L^{\infty}((0,T_0) \times \mathbb R)}  \\
&\leq 16 r \dfrac{\|n_0 \|_{L^1(\mathbb R)}}{\int_{\mathbb R} {n}_0(z){\rm d}z} \delta = 16r\delta. 
\end{aligned}
\]
Since $r \delta >0$ we obtain a contradiction. Hence $n$ may not take strictly negative values on $[0,T_0].$\qed
\end{proof}

\medskip

With this positivity result at-hand we show now quantitative bounds satisfied by the solution.
\begin{lemma} \label{lem_apbound}
We have:
\begin{itemize}
\item the pointwise bounds:
\begin{align}
 & n(t,x) \geq n_0(t,x) \exp\left[-\left(m(x) + \kappa  \max \left(\int_{\mathbb R} n_0(z){\rm d}z , \dfrac{r}{\kappa} \right) \right) t \right] ,\label{eq_positivite}\\
&  n(t,x) \leq n_0(x) + r\|\Gamma\|_{L^{\infty}(\mathbb R)}   \max \left(\int_{\mathbb R} n_0(z){\rm d}z , \dfrac{r}{\kappa} \right)t ,\label{eq_uniforme}
\end{align}
\item the averaged bounds:
\begin{align}
& \int_{\mathbb R} n(t,x) \leq \max \left(\int_{\mathbb R} n_0(z){\rm d}z , \dfrac{r}{\kappa} \right)  \label{eq_borneL1},\\
& \|n(t,\cdot)\|_{X_k} \leq  \|n_0\|_{X_k}  \exp(C_{k}t)  , \label{eq_borneXp}
\end{align}
where $C_{k}$ is a constant depending on the kernel $\Gamma$, $r$ and  $k$ only. 
\end{itemize}
\end{lemma} 
\begin{proof}
We recall that $n$ is a solution of 
\[ \partial_t   n = r T [ n ]  
-   \left( m  + \kappa \rho(t) \right) n .\]
In this lemma we assume that $n$ is non-negative. 
For the bound \eqref{eq_borneL1}, we first take the average in the equation, 
\[\partial_t  \int_{\mathbb R} n(t,x) {\rm d}x = r \int_{\mathbb R}T [ n ](x)  {\rm d}x
-   \int_{\mathbb R}m (x) n(t,x) {\rm d}x -\kappa \left(\int_{\mathbb R} n(t,x) {\rm d}x \right)^2. \]
Integrating in time between $0$ and $t$ gives 
\[ \int_{\mathbb R} n(t,x) {\rm d}x- \int_{\mathbb R} n_0(x) {\rm d}x \leq r \int_{\mathbb R}n (t,x)  {\rm d}x
-   \int_{\mathbb R}m (x) n(t,x) {\rm d}x -\kappa \left(\int_{\mathbb R} n(t,x) {\rm d}x \right)^2,\]
so that
\[ \int_{\mathbb R}(1+m(x)) n(t,x) {\rm d}x- \int_{\mathbb R} n_0(x) {\rm d}x \leq r \int_{\mathbb R}n (t,x)  {\rm d}x -\kappa \left(\int_{\mathbb R} n(t,x) {\rm d}x \right)^2.\]
Either $\int_{\mathbb R}n (t,x)  {\rm d}x \leq \int_{\mathbb R}n_0 (x)  {\rm d}x$ or $\int_{\mathbb R}(1+m(x)) n(t,x) {\rm d}x\geq \int_{\mathbb R}n (t,x)  {\rm d}x \geq \int_{\mathbb R}n_0 (x){\rm d}x \geq 0 $ so that $\int_{\mathbb R}n (t,x)  {\rm d}x \leq \dfrac{r}{\kappa}$ which gives \eqref{eq_borneL1}.
For the bound \eqref{eq_positivite}, we use ii) of Lemma \ref{lem_Tlip} so that $T[n] \geq 0$ and using \eqref{eq_borneL1}
\[\partial_t n \geq -(m+\kappa \rho(t))n\geq -\left(m+\kappa \max \left(\int_{\mathbb R} n_0(x){\rm d}x , \dfrac{r}{\kappa} \right) \right)n .\]
We deduce 
\[\partial_t \left(n \exp\left(\left(m+\kappa \max \left(\int_{\mathbb R} n_0(x){\rm d}x , \dfrac{r}{\kappa} \right)\right)t\right)\right)\geq 0\]
so 
\[n(t,x)\geq n_0(x) \exp\left(-\left(m+\kappa \max \left(\int_{\mathbb R} n_0(x){\rm d}x , \dfrac{r}{\kappa} \right)\right)t\right).\]
For the bound \eqref{eq_uniforme}, using that $n\geq 0$ and \eqref{eq_borneL1} we get 
\[\partial_t n \leq rT[n]\leq r\|\Gamma\|_{L^{\infty}(\mathbb R)}\rho(t)\leq r\|\Gamma\|_{L^{\infty}(\mathbb R)} \max \left(\int_{\mathbb R} n_0(x){\rm d}x , \dfrac{r}{\kappa} \right) \]
which gives \eqref{eq_uniforme}.
For the bound \eqref{eq_borneXp}, for a constant $C_k$ to be fixed, we compute 
\[\begin{aligned}
\partial_t \left(\exp(-C_k t)\|n(t,\cdot)\|_{X_k} \right) & = \int_{\mathbb R} w_k(x)  \exp(-C_k t)\partial_t  n(t,x) {\rm d}x-C_k\exp(-C_k t) \|n(t,\cdot)\|_{X_k}
\\ & \leq   \exp(-C_k t) r\int_{\mathbb R}  w_k(x)T [ n ](t,x){\rm d}x -C_k \exp(-C_k t)\|n(t,\cdot)\|_{X_k}
\end{aligned}
\] 
 We bound thanks to the change of variable $X=x-\dfrac{y+y'}{2}$ and the inequality $1 \leq w_k(X+\dfrac{y+y'}{2})\leq w_k(X) + w_k(y) + w_k(y')$, 
\[\begin{aligned}\int_{\mathbb R}  w_k(x)T [ n ](t,x){\rm d}x &\leq \int_{\mathbb R}\int_{\mathbb R} \int_{\mathbb R}w_k(X+\dfrac{y+y'}{2}) \Gamma \left(X\right)  n(t,y)\f{n(t,y')}{\rho(t)} {\rm d}y{\rm d}y' {\rm d}X
\\ &\leq 3 \|w_k \Gamma\|_{L^{1}(\mathbb R)} \|n(t,\cdot)\|_{X_k} \label{etape_decoupe}.
\end{aligned}\]
We then deduce, choosing $C_k: = 3r \|w_k \Gamma\|_{L^1(\mathbb R)}$,
\[\begin{aligned}
\partial_t \left(\exp(-C_k t)\|n(t,\cdot)\|_{X_k} \right) & \leq    3r \|w_k \Gamma\|_{\infty} \|n(t,\cdot)\|_{X_k}\exp(-C_k t) -C_k  \|n(t,\cdot)\|_{X_k} \exp(-C_k t)
\\ & \leq 0,
\end{aligned}
\] 
which gives the result.
{\qed}
\bigskip
\end{proof}

We point out that the pointwise bound {\eqref{eq_positivite}}  turns into a bound from below for the {integral} of $n(t,\cdot).$ Indeed, in case $n_0 \in \mathcal U_k \cap C_b(\mathbb R)$ is  nonnegative, there exists a compact interval $I$ such that 
\[
\int_I n_0(x){\rm d}x > 0.
\]
We have then:
\[
\int_{\mathbb R} n(t,x){\rm d}x \geq \int_{I} n(t,x){\rm d}x \geq  \exp \left[ - \left(\|m\|_{L^{\infty}(I)} + \kappa  \max \left(\int_{\mathbb R} n_0(x){\rm d}x , \dfrac{r}{\kappa} \right) \right) t \right] \int_{I} n_0(x){\rm d}x >0.
\]
In case $m$ is globally bounded,  we have then that the second item of the blow-up alternative never occurs for a nonnegative initial data in $\mathcal U_k$ so that solutions are global.

\subsubsection{Existence result for unbounded $m.$}
We proceed with extending the existence result to a possibly unbounded $m,$ we apply a compactness argument.  
From now on,  the initial data $n_0 \in \mathcal U_k \cap C_b(\mathbb R)$
is fixed.  Firstly,  we truncate $m$ and apply the previous construction yielding a sequence of approximate global solutions {$(n_{p})_{p \in \mathbb N}$}. We can then use the {\em a priori} bounds constructed in the latter lemma. These bounds enable to extract a subsequence converging in $L^{\infty}((0,T_0) \times \mathbb R) -w*$ (whatever $T_0 >0$) but is insufficient to obtain that the limit is continuous and satisfies  \eqref{eq_main}. For this purpose we provide the following a priori estimate:

\begin{lemma}
Let $R$ and $T_0$ positive and assume that $n$ is a solution to \eqref{eq_main} on $(0,T_0).$ Then, we have for
any $t\in [0,T_0]:$
\[
\sup_{(x,\tilde{x}) \in [-R,R]^2} |n(t,x) - n(t,\tilde{x})|
\leq 
\sup_{(x,\tilde{x}) \in [-R,R]^2} |n_0(x) - n_0(\tilde{x})| + 
C_{0}|x- \tilde{x}| \left( 1 + \|m\|_{C^1([-R,R])} \right) ,
\] 
where $C_0$ depends only on initial data, {\MH $T_0,$} $r$ and $\kappa$.
\end{lemma}
\begin{proof}
Let us define $\delta(t,x,\tilde{x})=n(t,x)-n(t,\tilde{x})$ so that 
\[\partial_t \delta = r (T[n] (t,x) -T[n](t,\tilde{x})) - (m(x)-m(\tilde{x}))n(t,x) -(m(\tilde{x})+\kappa\rho)\delta.\]
By Duhamel formula we get
\[\begin{aligned}
&\delta(t,x,\tilde{x})=\delta(0,x,\tilde{x})\exp\left(-m(\tilde{x})t-\int_{0}^{t} \kappa\rho(s) {\rm d}s \right)
\\& + \int_{0}^{t} (r(T[n](s,x)-T[n](s,\tilde{x}))-(m(x)-m(\tilde{x}))n(s,\tilde{x}))\exp\left(-m(\tilde{x})(t-s)-\int_{s}^{t} \kappa\rho(\tau) {\rm d}\tau \right){\rm d}s. 
\end{aligned}\]
We then deduce for $(x,\tilde{x})\in [-R,R]^2$ {\MH that,} 
\[|\delta (t,x,\tilde{x}) |\leq |\delta (0,x,\tilde{x})| + \int_{0}^{t} r|T[n](s,x)-T[n](s,\tilde{x})| {\rm d}s + |x-\tilde{x}|\|m\|_{C^1([-R,R])}\|n\|_{L^{\infty}} {t}.\]
Using that $x\rightarrow \exp(-x^2)$ is Lipschitz on $\mathbb{R}$, we get $|T[n](s,x)-T[n](s,\tilde{x})|\leq C |x-\tilde{x}| \rho(s)$ and we deduce the result with \eqref{eq_uniforme} and \eqref{eq_borneL1}.
{\qed}
\bigskip
\end{proof}

We emphasize that the above {\em a priori} estimate holds only on bounded intervals of $\mathbb R$ and is independent of the behavior of $m$ at infinity.  We can then apply this estimate to the approximate solutions $({n_p})$ yielding a local uniform bound.  Since we have also an {\em a priori} estimate
\eqref{eq_controldtn} on $(\partial_t {n_p})$ we infer that the sequence of approximate solutions $({n_p})$ is also compact in $C([0,T_0] \times [-R,R])$ for arbitrary positive $T_0,R.$ Eventually,  we combine the subsequent local strong convergence of a subsequence of approximate solutions {with} the uniform bounds in $L^{\infty}((0,T_0);X_k)$ to yield that the limit lies in $L^{\infty}(0,T_0;X_k)$ also and that this very subsequence converges also in $C([0,T_0];X_{\tilde{k}})$ for arbitrary $\tilde{k}  < k.$ In particular we have convergence in $C_{loc}([0,\infty),L^1(\mathbb R))$ and we can pass to the limit in the equations satisfied by approximate solutions. This entails that the limit is a solution to \eqref{eq_main}.

\subsection{Uniqueness of solutions}
 
To end up this section, we prove that, given $n_0 \in \mathcal U_k \cap C_b(\mathbb R),$ the solution we constructed above in case $m$ is not necessarily bounded is the unique solution in the sense of { Definition \ref{def_solution}}.  We restrict to the case {$k>3.$}

\medskip

For this, we denote $n$ the solution constructed above and $\tilde{n}$ a possible other solution with the same initial data $n_0 \in \mathcal U_k \cap C_b(\mathbb R)$.  We note that both solutions satisfy the conclusions of {Lemma \ref{lem_apbound}}. 
We pick then $k_0 \in  (3,k)$ and define
\[
\Delta_0(t) = \int_{\mathbb R} |n(t,x) - \tilde{n}(t,x)|^2 w_{k_0}(x){\rm d}x.
\]
Since $n$ and $\tilde{n}$ are bounded and continuous with values in $X_{\tilde{k}}$ whatever $\tilde{k} \in (k_0,k)$ we infer that 
$\Delta_0 \in C([0,\infty))$ and vanishes initially.  The key remark in the proof is then the following Gronwall inequality: 
\begin{lemma} \label{lem_gronwall}
Given $T_0 >0$ there exists a constant $C_{0}$ that depends only on initial data $n_0$, $k_0$, $T_0$, $r$, $\kappa$, $m$ and $\Gamma$  for which:
\[
\Delta_0(\tau) - \Delta_0(\sigma) \leq C_{0} \int_{\tau}^{\sigma} \Delta_0(s){\rm d}s.
\]
\end{lemma} 

Since $\Delta_0$ is continuous, we conclude by applying the Gronwall lemma that $\Delta_0$ vanishes identically on $[0,T_0]$ and then globally.  We end up this appendix by a proof of {\bf Lemma \ref{lem_gronwall}}. 

\begin{proof}
Let $\delta (t,x):= n(t,x) - \tilde{n}(t,x)$ and define $\Delta_0^R(t) =\displaystyle \int_{-R}^R |n(t,x) - \tilde{n}(t,x)|^2 w_{k_0}(x){\rm d}x$. Noticing that $\delta$ satisfies 
\[ \partial_t \delta(t,x)= r(T[n]-T[\tilde{n}])(t,x)-(m+\kappa \rho(t))\delta(t,x)-\kappa(\rho(t)-\tilde{\rho}(t))\tilde{n}(t,x), \]
we multiply this equation by $\delta w_{k_0}$ and integrate on $[\sigma,\tau]\times [-R,R]$ and get
\[\dfrac{1}{2}(\Delta_0^R(\tau)-\Delta_0^R(\sigma)) +\int_{\sigma}^{\tau} \int_{-R}^{R} (m+\kappa \rho(s)) \delta^2(s,x) w_{k_0}(x){\rm d}x{\rm d}s
=I_1+I_2 \]
where
\[I_1=-\kappa \int_{\sigma}^{\tau} \int_{-R}^{R} (\rho(s)-\tilde{\rho}(s)) \tilde{n}(s,x)\delta(s,x) w_{k_0}(x){\rm d}x{\rm d}s\]
and
\[I_2=r\int_{\sigma}^{\tau} \int_{-R}^{R}  (T[n]-T[\tilde{n}])(s,x)\delta(s,x) w_{k_0}(x){\rm d}x{\rm d}s.\]
Using that $\int_{\sigma}^{\tau} \int_{-R}^{R} (m+\kappa \rho(s)) \delta^2(s,x) w_{k_0}(x){\rm d}x{\rm d}s\geq 0$ we deduce 
\begin{align}
\Delta_0^R(\tau)-\Delta_0^R(\sigma) 
\leq 2(I_1+I_2). \label{eq_energy}
\end{align} 
For $I_1$, we bound $|\rho -\tilde{\rho}|$ using a Cauchy-Schwarz inequality and the fact that ${\f{1}{w_{k_0}}}\in L^{1}(\mathbb{R})$,
\[\begin{aligned}
|\rho(s) -\tilde{\rho}(s)| &\leq  \int_\mathbb{R} |n(s,x)-\tilde{n}(s,x)| {\dfrac{\sqrt{w_{k_0}(x)}}{\sqrt{w_{k_0}(x)}}} \,{\rm d}x \\ &\leq \left(\int_\mathbb{R} |n(s,x)-\tilde{n}(s,x)|^2 w_{k_0}(x){\rm d}x\right)^{1/2} \left(\int_\mathbb{R} \dfrac{1}{w_{k_0}(x)}{\rm d}x \right)^{1/2} \\ &\leq \Delta_0(s)^{1/2} \left(\int_\mathbb{R} \dfrac{1}{w_{k_0}(x)}{\rm d}x \right)^{1/2}
\\ & \leq C_{k_0}\Delta_0(s)^{1/2},
\end{aligned}\]
and we also get thanks to Cauchy-Schwarz inequality and using the bounds \eqref{eq_uniforme}, \eqref{eq_borneXp},
\[ \begin{aligned}
\int_{-R}^R \tilde{n}(s,x) \delta(s,x) w_{k_0}(x){\rm d}x &\leq \left(\int_\mathbb{R} \tilde{n}(s,x)^2 w_{k_0}(x){\rm d}x\right)^{1/2} \left(\int_\mathbb{R} \delta(s,x)^2 w_{k_0}(x){\rm d}x \right)^{1/2} 
\\ &\leq  \|n(s,\cdot) \|_{L^{\infty}(\mathbb{R})}^{1/2} \|n(s,\cdot)\|_{X_k}^{1/2}  \Delta_0^{1/2}(s)
\\ & \leq C_{k_0, n_0, \Gamma,r,\kappa}  \Delta_0(s)^{1/2} ,
\end{aligned} \]
so that
\begin{align}
\label{estim_I1}
|I_1| \leq C_{k_0, n_0, \Gamma,r,\kappa} \int_{\sigma}^{\tau}\Delta_0(s) {\rm d}s.  
\end{align}
 
For $I_2$, we first have by a standard Cauchy-Schwarz inequality: 
\[
I_2 \leq r\int_{\sigma}^{\tau} \left(\int_{\mathbb R}  |T[n] -T[\tilde{n}]|^2 w_{k_0} \right)^{\frac 12}  \Delta_0^{\frac 12}.
 \]
However,  introducing $\ell \in (1,k_0-1)$ to be fixed later on, recalling Lemma \ref{lem_Tlip} with \eqref{eq_Cbw} and remarking that $\|\cdot\|_{{X_\ell}} \leq \|\cdot\|_{{X_k}}$, we infer that:
\[
\begin{aligned}
\int_{\mathbb R}  |T[n] -T[\tilde{n}]|^2 w_{k_0}  
& \leq \| w_{\ell}(T[n] -T[\tilde{n}])\|_{C_b(\mathbb R)} \|T[n] -T[\tilde{n}]\|_{X_{k_0-\ell}} \\
& \leq K \left( \f{\| n(s,\cdot) \|_{X_{k}}}{\left|\int_{\mathbb{R}} n(s,z){\rm d}z\right|}\f{\| \tilde{n}(s,\cdot) \|_{X_k}}{\left|\int_{\mathbb{R}}\tilde{n}(s,z){\rm d}z\right|} \right) \left( \|n\|_{X_{k}} + \|\tilde{n}\|_{X_{k}}\right) \|n- \tilde{n}\|_{X_{k_0-\ell}}. 
\end{aligned}
\]
Then, by Lemma \ref{lem_apbound} there is a constant $C_{0}$ depending only on initial data $n_0$, $r$, $\kappa$, $\Gamma$, $m$ and $T_0$ so that:
\[
\sup_{s \in (0,T_0)} K \left( \f{\| n(s,\cdot) \|_{X_k}}{\left|\int_{\mathbb{R}} n(s,z){\rm d}z\right|}\f{\| \tilde{n}(s,\cdot) \|_{X_k}}{\left|\int_{\mathbb{R}}\tilde{n}(s,z){\rm d}z\right|} \right) \left( \|n\|_{X_{k}} + \|\tilde{n}\|_{X_{k}}\right) \leq C_{0},
\]
using that there exists a bounded interval $I$ such that $\int_{\mathbb{R}} n_0(z){\rm d}z>0$ so that $m$ is bounded and 
\begin{align*}
\int_{\mathbb{R}} n(s,z){\rm d}z & \geq  \int_{I} n(s,z){\rm d}z >\int_{I} n_0(z) \exp\left[-\left(m(z) + \kappa  \max \left(\int_{\mathbb R} n_0(x){\rm d}x , \dfrac{r}{\kappa} \right) \right){\MH T_0}\right] {\rm d}z \\
& > C_{0}\int_{\mathbb{R}} n_0(z){\rm d}z.
\end{align*}
Then a Cauchy-Schwarz inequality entails that, for arbitrary ${\mu} >0$ there is $C_{{\mu}}$ such that:
\[
 \|n- \tilde{n}\|_{X_{k_0-\ell}} \leq  C_{{\mu}} \left( \int_{\mathbb R} |n-\tilde{n}|^2 w_{2(k_0-\ell) + (1+{\mu})} \right)^{\frac 12} 
\]
Since $k_0 >3$ we can choose $\ell$ and ${\mu}$ so that $2(k_0-\ell)+ (1+{\mu}) = k_0$ and we obtain finally that:
\begin{equation}
I_2  \leq C_{0} \int_{\sigma}^{\tau} \Delta_0(s) {\rm d}s. \label{estim_I2}
\end{equation}
with a constant $C_{0}$ depending on the chosen $\ell,\varepsilon,$ initial data and $T.$
We conclude combining \eqref{eq_energy}, \eqref{estim_I1}, \eqref{estim_I2} and sending $R$ to $+\infty$.
{\qed}

\end{proof}

\bigskip
\noindent \textbf{Acknowledgements}: 
M.H.  acknowledges support of the Institut Universitaire de France. This paper was partly realised while M.H. was benifiting a "subside \`a savant" from Universit\'e Libre de Bruxelles.  M.H.  would like to thank  the mathematics department at ULB for its hospitality. 
S.M. has been supported by  the ANR project DEEV ANR-20-CE40-0011-01 and the Chair ``Mod\'elisation Math\'ematique et Biodiversit\'e" of Veolia Environnement-Ecole Polytechnique-Museum National d'Histoire Naturelle-Fondation X.
\bibliographystyle{plain}

 \end{document}